\documentclass[final,reqno,onefignum,onetabnum]{siamltex}

\usepackage[letterpaper,top=2in, bottom=1.5in, left=1in, right=1in]{geometry}

\usepackage{epsfig,epsf,fancybox}
\usepackage{amsmath}
\usepackage{mathrsfs}
\usepackage{amssymb}
\usepackage{amsfonts}
\usepackage{graphicx}
\usepackage{color}
\usepackage[linktocpage,colorlinks,linkcolor=blue,anchorcolor=blue,citecolor=magenta,urlcolor=cyan,hypertexnames=false]{hyperref}
\usepackage{boxedminipage}
\usepackage{stmaryrd}
\usepackage{multirow}
\usepackage{booktabs}
\usepackage{relsize}
\usepackage{accents}
\usepackage{cite}
\usepackage{float}
\usepackage{algorithm}
\usepackage{algpseudocode}
\usepackage{breqn}
\usepackage{lineno}
\usepackage{mathtools}
\usepackage{xcolor}
\usepackage{enumerate}
\usepackage{subfigure}

\usepackage{stmaryrd}
\SetSymbolFont{stmry}{bold}{U}{stmry}{m}{n}

\def\R{\mathbb{R}}

\newcommand{\zz}{^{\top}}

\newcommand{\TT}{{\scriptscriptstyle{\, \textnormal{T}}}}

\newcommand{\lb}{\lambda}
\newcommand{\sign}{\mathrm{sign}}
\newcommand{\LL}{\mathcal{L}}
\newcommand{\sn}{{{\mathcal{S}^n}}}
\newcommand{\snA}{{{\mathcal{S}_A^n}}}

\newcommand{\tr}{\operatorname{tr}} 

\newcommand{\st}{\textnormal{s.t.}}

\newcommand{\revise}[1]{\textcolor{black}{#1}}

\usepackage{mathtools}

\usepackage[normalem]{ulem} 

\newcommand{\bc}{\begin{center}}
\newcommand{\ec}{\end{center}}

\newcommand{\bdm}{\begin{displaymath}}
\newcommand{\edm}{\end{displaymath}}

\newcommand{\beq}{\begin{equation}}
\newcommand{\eeq}{\end{equation}}

\newcommand{\bfl}{\begin{flushleft}}
\newcommand{\efl}{\end{flushleft}}

\newcommand{\bt}{\begin{tabbing}}
\newcommand{\et}{\end{tabbing}}

\newcommand{\beqn}{\begin{eqnarray}}
\newcommand{\eeqn}{\end{eqnarray}}

\newcommand{\beqs}{\begin{align*}} 
\newcommand{\eeqs}{\end{align*}}  

\newtheorem{remark}{Remark}[section]

\begin{document}
	
		\title{A  {Graph-Partitioning Based} Continuous Optimization Approach to Semi-supervised Clustering Problems}
	
\author{Wei Liu\thanks{Department of Mathematical Sciences,
		Rensselaer Polytechnic Institute, US ({liuwei175@lsec.cc.ac.cn})}
	\and Xin Liu\thanks{State Key Laboratory of Mathematical Sciences, Academy of Mathematics and
		Systems Science, Chinese Academy of Sciences, and University of Chinese Academy of Sciences,
		China ({liuxin@lsec.cc.ac.cn})}
	\and Michael K. Ng\thanks{Department of Mathematics,
		Hong Kong Baptist University, Hong Kong, China ({michael-ng@hkbu.edu.hk})}
	\and Zaikun Zhang\thanks{School of Mathematics, Sun Yat-sen University, Guangzhou, China ({zhangzaikun@mail.sysu.edu.cn})}
}
	
	\date{\today}
	
	\maketitle
	\begin{abstract}
		Semi-supervised clustering is a basic problem in various applications. Most existing methods require knowledge of the ideal cluster number, which is often difficult to obtain in practice. Besides, satisfying the must-link constraints is {another} major challenge {for} these methods. In this work, we view the semi-supervised clustering task as a  partitioning problem on a graph associated with the given dataset,  where the {similarity} matrix  includes a scaling parameter 
			to reflect the must-link constraints. Utilizing a relaxation technique, we formulate the graph partitioning problem into a continuous optimization model that does not require the exact cluster number, but only an overestimate of it. 
			We then propose a block coordinate descent algorithm to efficiently solve this model, and 
			establish its
			convergence result. Based on the obtained solution,  we can construct the clusters
			that theoretically meet the must-link constraints  under mild assumptions.
			Furthermore, we verify the effectiveness and efficiency of our proposed method
				through comprehensive numerical experiments.
	\end{abstract}

	
	\section{Introduction}\label{sec:intro}
	As a widely utilized approach in unsupervised or semi-supervised learning, clustering is a basic tool  for grouping similar objects together and separating dissimilar ones~\cite{jain1999data,jarvis1973clustering}.
	{It has applications in data mining~\cite{han2011data}, statistical analysis~\cite{von2007tutorial},  
		image analysis~\cite{du1999centroidal,arbelaez2011contour,felzenszwalb2004efficient,guo2003clustering}, and
		computer graphics~\cite{Weiss_1999},
		etc.}
	
	A classical unsupervised clustering method is $k$-means~\cite{ball1965isodata,macqueen1967some,steinhaus1956division}, 
	\revise{which partitions the given data by assigning points to the nearest centroid, and the resulting clusters are separated by decision boundaries that are linear hyperplanes in the case of Euclidean distance}
	~\cite{aurenhammer1991voronoi,bankman2008handbook}. Consequently, {many clustering problems cannot be handled by} $k$-means  (see examples in~\cite{ng2001spectral}).  
	In contrast, spectral clustering approaches~\cite{arias2017spectral,dhillon2004unified,hagen1992new,scholkopf1998nonlinear,shi2000normalized,von2007tutorial} can tackle problems where the underlying clusters are not linearly separable~\cite{arias2017spectral}. {For example, standard Spectral Clustering (SC)~\cite[Section 4]{von2007tutorial} consists of two steps.  The first step is to solve a continuous optimization problem
		\begin{equation}
			\label{eq:SC}
			\min_{D\zz   D=I_{k}} \operatorname{tr}\left(D\zz \mathcal{L}({A}){D}\right),
		\end{equation} where $k$ is the number of clusters, $I_k$ is the $k\times k$ identity matrix, $A$ is a nonnegative matrix indicating the similarity between the data points (called the similarity matrix in~\cite{von2007tutorial}), and $\mathcal{L}$ denotes the Laplacian operator defined by
		$$
		\mathcal{L}(A):=\mathrm{Diag}(A e_n)-A.
		$$ Here, $\mathrm{Diag}(z)$ refers to the diagonal matrix whose diagonal vector is $z$,  and $e_n\in \R^n$ represents the vector with all elements being 1. The second step is to extract the {resulting}
		clusters from~$D$ via a post-processing step such as $k$-means.  
A drawback of SC is its {heavy} dependence on $k$, which equals the number of clusters that SC {outputs}. Consequently, SC can determine the correct clusters only if we know the correct number of clusters beforehand. 

Semi-supervised clustering utilizes supervisory information to improve the clustering performance~\cite{bruzzone2006novel,castelli1995exponential,qin2019research}. 
The must-link constraint is one representative kind of supervisory information, which indicates that two data points must belong to the same cluster.
It has been shown recently that {enforcing} {correct} must-link {constraints}  helps prevent the occurrence of undesired clusters \cite{khoreva2014learning}. 
The constrained $k$-means and heuristic $k$-means~\cite{bair2013semi,miller2009geographic} are two semi-supervised clustering methods that 
can  
{deal with} supervisory information in the clustering process.
However, the same as $k$-means, these methods 
{are only suitable for cases in which}
the underlying clusters are linearly separable. 
To address this issue, semi-supervised spectral clustering has been proposed~\cite{bair2013semi,belkin2004regularization, jia2018semi,kamvar2003spectral,khoreva2014learning,von2007tutorial,wang2010flexible,yang2014unified, zhou2003learning,zhu2005semi2,zhu2005semi1}. As introduced in~\cite{jia2018semi}, the existing semi-supervised spectral clustering methods can be {divided} into two main categories. 
{Both of them modify {the first step} 
	of the standard SC  {using the supervisory information}.}  
	The first category of {these} methods {incorporates} such  information {into}  matrix $A$. For example,~\cite{kamvar2003spectral,von2007tutorial} assign specific values to $A_{ij}$ based on the presence of labels between data points $x_i$ and $x_j$.  
	The second category integrates {the supervisory information} directly into the variable~$D$. For example,~\cite{li2009constrained,wang2010flexible,yang2014unified} impose certain constraints with respect to $D$ on problem~\eqref{eq:SC}. However, {none of these methods}  
	{guarantees that the must-link constraints can be completely satisfied by}
	the resulting clusters~\cite{khoreva2014learning}, 
	and they also require 
	{a pre-known precise number of clusters $k$.} {There exist other continuous optimization approaches to semi-supervised clustering,  such as Symmetric NMF (SymNMF)~\cite{kuang2012symmetric} and Semi-supervised NMF (SemiNMF)~\cite{yang2014unified}, which solve matrix factorization problems to obtain clusters. Similar to SC, these approaches still depend on the number $k$ of clusters.}
	 
	{In this paper, we aim}
	to design a clustering method that does not {require}  the precise cluster number $k^*$. 
To this end, we view the clustering task as a  partitioning problem on a {weighted} graph associated with the given dataset.  
{We pursue to}
{obtain as many connected subgraphs as possible by selectively removing a few edges from the weighted graph with as small a sum of weights as possible.}
By introducing a model parameter to make a trade-off between the above-mentioned {two goals}, 
we propose a continuous optimization model for semi-supervised spectral clustering. 
{This} 
model only needs an {overestimate}  
of the ideal number of clusters.  
{Exploiting} the block structure of the proposed model, we develop an efficient block coordinate descent algorithm that guarantees finite convergence to a blockwise minimizer. 
Our model contains  the {similarity} matrix  of
the resulting graph   as a variable.
{Hence}, we can determine the clusters via a simple search algorithm on {this} {similarity} matrix. In contrast, {the mentioned SC-based and NMF-based methods} require multiple steps {including invoking an additional clustering algorithm} to determine the {resulting} clusters.  
{{Last but not least, theoretical analysis guarantees that} our approach {satisfies} the must-link {constraints} under mild conditions.} These {properties} make our method {more} practical and robust {than existing
approaches.}  
{Extensive numerical experiments on {synthetic datasets and document clustering}   demonstrate the {advantages} of our approach.}

\subsection{Notations}\label{sec:nota}
Suppose that $X=\{x_1,x_2,\ldots,x_n\}\subset \R^m$ is a dataset to be clustered.  We use $A\in\R^{n\times n}$ to represent the similarity matrix, which indicates the pairwise similarities between all pairs of data points in $X$. We assume that $A$ is nonzero and nonnegative in this paper. {Given the data $X$ and the  matrix $A$, we define the corresponding graph $\mathcal{G}(X, A)$, where the vertex set is $X$, and the edge set  is $\{(i, j) \mid A_{ij} > 0\}$,} {with $A$ working as the weighted adjacency matrix of the graph~\cite{von2007tutorial}.}

{We denote} $\sn=\{S\in\R^{n\times n}\mid S=S\zz\}$ and $\mathcal{S}_{0}^{n}=\{S \in \sn \mid S \geq 0, \operatorname{diag}(S)=0\}$, where {the inequality} $S\geq 0 $ 
{means that the entries}
of $S$ are not less than 0. Given a matrix $M\in\sn$, we use $\operatorname{supp}(M)$ to represent the index set of its nonzero {entries} and define $\mathcal{S}^n_M=\left\{S \in \sn \mid \operatorname{supp}(S) \subset \operatorname{supp}({M})\right\}$.  
We denote {$\lb_1(M)\le\lb_2(M)\le\cdots\le\lb_n(M)$ as}
the eigenvalues of $M$. 

Given matrices $M$ and $N$ with the same size, $M \ge N$ signifies that all entities of $M-N$ are nonnegative, and $M\succeq N$  means that $M-N$ is positive semi-definite.
$M\circ N$ stands for the Hadamard product of $M$ and $N$. 
$\lceil{M}\rceil$ is a matrix of the same size as $M$, with $\lceil{M}\rceil_{ij}=\lceil{M}_{ij}\rceil$ and $\lceil x\rceil$ referring to the smallest integer that is at most $x$. 
sign($M$) is also a matrix of the same size as $M$, with sign($M)_{ij}$ = sign($M_{ij}$), where we define sign(0) = 0. 
We use $\{0, 1\}^{n\times n}$ to represent the set of all $n\times n$ matrices whose elements are either 0 or 1. $\left\langle \cdot, \cdot \right\rangle$ refers to the Frobenius inner product of two matrices. 

\subsection{Organization}

The rest of this paper is organized as follows. In Section~\ref{sec:2}, 
we construct our continuous optimization model~\eqref{eq:model},  
and investigate its theoretical properties. Section~\ref{sec:3} develops  a block coordinate descent algorithm to solve the optimization problem in~\eqref{eq:model}, and analyzes its convergence. 
Then, we show some numerical experiments in Section~\ref{sec:4}. Finally, concluding remarks are presented in the last section.

\section{A Continuous optimization approach to semi-supervised clustering}
\label{sec:2}
{In this section, we construct our Continuous Optimization approach to Semi-Supervised Clustering (COSSC), which reformulates clustering as a graph partitioning task. 
It involves two main steps: first, solving a continuous model that integrates must-link constraints using an efficient algorithm; second, deriving clusters through a graph-based search algorithm.
	We  analyze the properties of the proposed model, and demonstrate its ability to satisfy the must-link constraints in the resulting clusters.}

\subsection{Model formulation and cluster identification}\label{sec:focus}
Suppose that the dataset $X$, the  similarity  matrix $A$, and the label set $\mathcal{J}$ 
are given,
where 
\begin{equation*}
\mathcal{J} : =\bigg\{(i,j) \mid \text{ there is a must-link constraint between $x_i$ and $x_j$}\bigg\}.
\end{equation*} 
In this paper, we assume that $A_{i j} > 0$ for all $(i, j) \in \mathcal{J}$. 
{Based on the matrix $A$ and the label set $\mathcal{J}$,}
we construct a modified  matrix $\bar{A}$ with
\begin{equation}
\label{eq:Aconstruct}
\bar{A}_{ij}:=
\left\{
\begin{aligned}
&p A_{ij}, \text{  if } (i,j) \in \mathcal{J},\\
&A_{ij}, \,\,\,\text{  otherwise},
\end{aligned}\right.
\end{equation}
where  
$p\geq 1 $ is a weight {representing the strength of the connection between} $x_i$ and $x_j$, reflecting the must-link constraints. Note that $A$ and $\bar{A}$ have the same support set, leading to the following lemma. 

\begin{lemma}
	Let $\bar{A}$ be defined as in \eqref{eq:Aconstruct}.
	We then have $\snA=\mathcal{S}^n_{\bar{A}}$ and
	\begin{align}
		\label{eq:rank22}
		\rank(\LL (A))=\rank(\LL (\bar A)).
	\end{align}
\end{lemma}

\begin{proof}
	{
	Since  
	$\supp(A)=\supp(\bar{A})$, we have $\snA=\mathcal{S}^n_{\bar{A}}$. To prove the rank equality \eqref{eq:rank22}, 	we consider the graphs with $A$ and $\bar{A}$ being the adjacency matrices. Note that the connected subgraphs of the two graphs are identical, since they are determined by the support sets of $A$ and $\bar{A}$.
	According to~\cite[Theorem 2.1]{Mohar_1991}, the numbers of connected subgraphs  are $n- \rank(\LL (A))$ and $n- \rank(\LL (\bar A))$, respectively.
	Thus, we have the desired result.
	}
\end{proof}

Next, we approach the clustering task as a graph partitioning problem. The goal is to obtain   many connected subgraphs  
by selectively removing  {a few} edges from $\mathcal{G}(X, \bar{A})$ {while minimizing the total weight of the removed edges}.  
{To achieve this, }we define a {binary} matrix $Z\in \mathcal{S}^n_{\bar A} \cap\{0, 1\}^n = \snA \cap\{0, 1\}^n$ to indicate which edges {are to} be removed, with
\begin{equation*}
Z_{ij}:=\left\{
\begin{split}
0, & \quad \text{ if edge $(i,j)$ {is to} be removed from graph }\mathcal{G}(X, {\bar{A}}), \\
1, & \quad \text{ else.}
\end{split}
\right.
\end{equation*}
Note that the number of connected subgraphs in $\mathcal{G}(X, {\bar{A}\circ Z})$ is $n-\rank(\LL(\bar{A}\circ Z))$, which equals $n-\rank(\LL(A\circ Z))$ similar to~\eqref{eq:rank22}. In addition, $\sum_{i<j}\bar{A}_{ij}-{1\over 2}\tr(\bar{A}Z)$ is the {total} weight of the edges {to be removed}.
Therefore, we can formulate {the aforementioned goal} as an optimization problem:
\begin{equation}
\label{eq:modelORIGIN}
\begin{split}
\min_{Z} \;\; & (\rank(\LL(A\circ Z)) -n + d) + \beta \left(2\sum_{i<j}\bar{A}_{ij}-\tr(\bar{A}Z)\right) \\
\st \;\; & Z \in\snA \cap \{0, 1\}^{n\times n}, 
\end{split}
\end{equation}
where $\beta>0$ is a balancing parameter. Here, we add $d$ to the objective function, where $d$ is an overestimate of the ideal number $k^*$ of clusters and we assume that $  d\ll n$. 
The purpose of adding $d$ will be clear later.
	This flexibility in 
	$d$ distinguishes our approach from traditional spectral methods such as~\eqref{eq:SC},  because we do not require $d$ to be {exactly} $k^*$. {Indeed, {we will demonstrate}} in our numerical tests that our method is insensitive to the value of $d$ {as long as it is not too large}, consistently producing the ideal number of clusters with {appropriately} selected parameters {$\beta$ and $p$; see Remark~\ref{rem:d} for further discussion.}}

{Due to} the {discontinuous nature of the function} $\rank(\LL(A\circ Z)) $ and the {binary} set $\{0, 1\}^{n\times n}$,  problem~\eqref{eq:modelORIGIN} is a mixed-integer programming {which is still difficult to tackle.}
{{To address this difficulty}, we will approximate the discontinuous objective function by a continuous one and relax the binary set,
thereby yielding a continuous optimization model.} 

	In the low-rank case,  function ${\tr}(\cdot)$ is commonly used as  a convex relaxation of $\rank(\cdot)$~\cite{fazel2001rank}, but unfortunately, the matrix $\LL(A\circ Z)$ is not low-rank, which prevents the direct application of this relaxation technique. {To address this issue,  we introduce an auxiliary matrix $\widetilde{H}$ that helps create a low-rank matrix.}
Assume that the eigendecomposition of $\LL(A \circ Z)$ is given by \revise{$Q \Lambda Q^\top$}, where $Q$ is an orthogonal matrix, and the diagonal entries of $\Lambda$ are arranged in ascending order. Let $\widetilde{H}$ be the matrix consisting of the first $d$ columns of $Q$.
We then have
\begin{equation}
\rank( \LL(A\circ Z)\widetilde{H}\widetilde{H}\zz) \geq
 \rank(\LL(A\circ Z))-(n-d). 
\end{equation}
In addition, the matrix~\(\LL(A \circ Z)\widetilde{H}\widetilde{H}\zz\in\R^{n\times n}\) is symmetric, and has a rank of at most \(d\ll n\), making it a symmetric low-rank matrix.
Hence
\begin{align}
\rank( \LL(A\circ Z)\widetilde{H}\widetilde{H}\zz) \approx	\tr(\LL(A \circ Z)\widetilde{H}\widetilde{H}^\top)=
\min_{H^\top H = I_d} \tr(H^\top \LL(A \circ Z) H).
\end{align}
Given these relationships, we replace  $\rank(\LL(A\circ Z)) -(n-d)$ in  \eqref{eq:modelORIGIN} with     $\min_{H\zz H=I_d} \tr({H}\zz \LL(A\circ Z){H})$, obtaining
\begin{equation}
	\label{eq:modelORIGIN2}
	\tag{MIP}
	\begin{split}
		\min_{Z, H} \;\; & f(Z,H):=\tr\left(H\zz \LL(A\circ Z)H\right) - \beta\tr(\bar{A}Z) \\
		\st \;\,\, & Z \in\snA \cap \{0, 1\}^{n\times n},\,\, H\zz  H = I_{d},\,\, H\in\R^{n\times d}.
	\end{split}
\end{equation}%
Furthermore, we relax $\snA \cap \{0, 1\}^{n\times n}$ to $\snA \cap [0, 1]^{n\times n}$ and arrive at the continuous optimization problem:
\begin{equation}
\label{eq:model}
\tag{CP}
\begin{split}
\min_{Z,H} \;\; & f(Z,H) \\
\st \;\; & Z \in\snA \cap [0, 1]^{n\times n},\,\, H\zz  H = I_{d},\,\, H\in\R^{n\times d},
\end{split}
\end{equation}
{which is our continuous optimization model for semi-supervised clustering.} 
Comparing \eqref{eq:modelORIGIN} with~\eqref{eq:model}, the latter relaxes rank to trace and $\{0,1\}^{n\times n}$ to $[0,1]^{n\times n}$. It is worth noting here that the objective function $f(Z,H)$ is  a linear function with respect to $Z$. 

As will be discussed in Section~\ref{sec:3}, model~\eqref{eq:model} exhibits a block structure, which facilitates its solution using a block coordinate descent algorithm.
{Once} we {get} a solution $Z^*$ to problem~\eqref{eq:model}, the corresponding graph $\mathcal{G}(X, {\bar{A}\circ Z^*})$ is obtained. 
We finally identify the clusters in $X$ 
by running a search algorithm on $\mathcal{G}(X, {\bar{A}\circ Z^*})$,
such as Breadth-First-Search (BFS)~\cite{xu2003document}.
Notably, we can obtain the same clusters by executing a search algorithm on $\mathcal{G}(X,  Z^*)$, since $\supp(Z^*)= \supp(\bar{A}\circ Z^*).$

\subsection{Model analysis}

In this subsection, 
we first {expose} the relationships between 
our proposed models~\eqref{eq:model} and~\eqref{eq:modelORIGIN2}, 
\revise{in particular, the inclusion relation of their solution sets.}
We also provide sufficient conditions under which  the must-link constraints can be satisfied.

\begin{theorem}
\label{thm:x}
\textnormal{(a)} If $(Z^*, H^*) $ is a global minimizer of problem~\eqref{eq:model}, then $(\lceil Z^* \rceil, \, H^* ) $ is also a global minimizer of problem~\eqref{eq:modelORIGIN2}. 

\textnormal{(b)} If $(Z^*, H^*) $ is a global minimizer of problem~\eqref{eq:modelORIGIN2}, then $(Z^*, \, H^* ) $ is also a global minimizer of problem~\eqref{eq:model}.
\end{theorem}

\begin{proof}
(a) Notice that $f$ is a linear function with respect to $Z$. For any $i, j\in\{1,2,\ldots,n\}$, if $Z_{ij}^*\in(0,1)$, then $(\nabla_Z f(Z^*, H^*))_{ij}=0$;
if $Z_{ij}^*\in\{0,1\}$, then $\lceil{Z^*_{ij}}\rceil= Z^*_{ij}$. {Combining the above two cases,} it holds  that 
$$
f(\lceil{Z^*}\rceil, \, H^* )-f({Z^*}, \, H^* ) =\langle \nabla_Z f(Z^*, H^*),  \lceil Z^*\rceil- Z^*\rangle =0.
$$ This indicates that $(\lceil{Z^*}\rceil, \, H^* ) $ is also a global minimizer of problem~\eqref{eq:modelORIGIN2}. 

(b) The second part follows directly, as every global minimizer of \eqref{eq:modelORIGIN2} is also feasible for~\eqref{eq:model}.
\end{proof}

The above theorem reveals the relationship between the global solution set  of model~\eqref{eq:model} and that of model~\eqref{eq:modelORIGIN2}.   
We now present two theorems concerning the global minimizers of problem~\eqref{eq:model}. The proofs of these two theorems are given in Appendices~\ref{appen:1} and~\ref{appen:2}, respectively.

\begin{theorem}\label{thm:thmor}
Let $d\ge n-\rank(\LL({A}))$. Each global minimizer $(Z^*, H^*)$ of problem~\eqref{eq:model} satisfies exactly one of the following two conditions:

\textnormal{(a)}\,\,$\rank(\LL({A} \circ Z^*) ) > n-d$; 

\textnormal{(b)} $\rank(\LL({A} \circ Z^*) ) = n-d$, and $Z^* \in \{0, 1\}^{n\times n}$.

\end{theorem}

\begin{remark}\label{rem:d}
\textit{By~\cite[Theorem 2.1]{Mohar_1991}, $n-\rank(\LL(\bar{A}\circ Z^*))$ is the number of connected subgraphs in $\mathcal{G}(X, {\bar{A}\circ Z^*})$, which equals the number of the output clusters.
By Theorem~\ref{thm:thmor}, this number {is at most} $d$. 
Moreover, we observe in our numerical experiments that 
this number often equals 
the ideal number of clusters $k^*$ 
when $d\ge k^*$  and is close to $k^*$ 
(see Sections~\ref{sec:graphdiffn} and~\ref{sec:docudiffn}).
}
\end{remark}

\begin{theorem}
\label{thm:bigb}
Given $n-\rank(\LL({A})) \le d <n $, {the following statements hold.} 

\textnormal{(a)} ~If $\beta>1$, then each global minimizer $(Z^*, H^* ) $ of problem~\eqref{eq:model} satisfies $Z^* =\sign(A)$.

\textnormal{(b)}~There exists $\overline{\beta}>0$ such that for any {$\beta<\overline{\beta}$,} each global minimizer $(Z^*, H^* )$ of problem~\eqref{eq:model} satisfies 
$\rank(\LL({A} \circ Z^*) ) = n-d$ and $Z^* \in \{0, 1\}^{n\times n}$.
\end{theorem}

\begin{remark}
\label{rem:equiv}
\textit{\textnormal{(a)}
Let $\beta>1$.
From Theorem \ref{thm:bigb}\textnormal{(a)}, we conclude that 
$$\mathcal{G}(X,\bar{A}\circ Z^*) =  \mathcal{G}(X,\bar{A}\circ  \sign(A)) =  \mathcal{G}(X,\bar{A}\circ  \sign(\bar A))= \mathcal{G}(X,\bar{A}).$$ This means that our model does not remove any edge from the graph $\mathcal{G}(X,\bar{A})$. {Therefore, it is not appropriate to take $\beta>1$, unless the clusters associated with the similarity matrix $\bar A$ are already very good.}}

\textit{
\textnormal{(b)} Let $0 <\beta < \overline{\beta}$. Theorem \ref{thm:bigb}\textnormal{(b)} ensures that $\rank(\LL({A} \circ Z^*) ) = n-d$. 
Then, we have
$
\tr \left((H^*)\zz \LL(A\circ Z^*) H^*\right) = 0$ by applying the Rayleigh-Ritz theorem.
Hence, problem~\eqref{eq:model} has the same global minimizers as the following problem:
\begin{equation}
	\label{eq:model2}
\begin{split}
	\min_Z \,\, &  - \tr(\bar{A}Z)\\
	\st \,\, & Z \in \{0,1\}^{n\times n} \cap \snA,\,\rank(\LL( A\circ Z) ) = n-d.
\end{split}
\end{equation}
Similar to~\eqref{eq:rank22}, we have $\rank(\LL({\bar A} \circ Z^*) ) = \rank(\LL( A\circ Z^*) ) = n-d$, since $\supp({\bar A} \circ Z^*)=\supp({A} \circ Z^*)$.
This means that the graph $\mathcal{G}(X, {\bar A} \circ Z^*)$ has exactly
$d$	
connected subgraphs. 
Thus, the goal of the model \eqref{eq:model2} is to obtain exactly
$d$	
connected subgraphs by selectively removing a few edges from $\mathcal{G}(X, {\bar{A} })$ 
so that the total weight of the removed edges is minimized.}
\end{remark}

Provided that $\beta p>2$, the following theorem provides sufficient conditions on a feasible point $(Z^*, H^*)$ of problem \eqref{eq:model} so that
the associated graph $\mathcal{G}(X, \bar{A}\circ Z^*)$
meets the must-link requirement. 
\begin{theorem}
\label{thm:mustlink2}
Let the label set $\mathcal{J}$ and $0\leq \epsilon<2\min_{(i,j)\in\mathcal{J}}A_{ij}$ be given. {Suppose that $\beta p > 2$, and that $\left(Z^*, H^*\right)$ satisfies $$Z^*\in \snA \cap \{0, 1\}^{n\times n}, \,\, (H^*)\zz H^*=I_d, \text{  and }f\left(Z^*, H^*\right) -  \min_{Z \in \snA \cap[0,1]^{n \times n}} f\left(Z, H^*\right) \leq \epsilon.$$ 
Then we have $Z^*_{{i}{j}}= 1$ for all $({i},{j})\in \mathcal{J}$}. 
\end{theorem}

\begin{proof}
Assume, for contradiction, that there exists a pair $(\bar{i},\bar{j})\in \mathcal{J}$ such that $Z^*_{\bar{i}\bar{j}} =0$. 
Construct $\bar{Z}$ by
\begin{equation}
\bar{Z}_{ij}:=
\left\{\begin{aligned}
	&1, \quad\text{  if } (i,j)=(\bar{i},\bar{j}) \text{  or } (i,j)=(\bar{j},\bar{i}),\\
	& {Z}^*_{ij},\text{  otherwise}.
\end{aligned}\right.
\end{equation}
Then $\bar{Z}$ lies in $\snA \cap\{0,1\}^{n \times n}$, since $Z^* \in \mathcal{S}_A^n \cap\{0,1\}^{n \times n}$ and  $\mathcal{J} \subset \supp(A)$. 
Let $Q:=H^*(H^*)\zz$. As is given by equation~\eqref{eq:derivef}, the gradient $\nabla_Z f(Z, H^*)$ satisfies  
\begin{align}
	\label{eq:gradient}
	(\nabla_Z f(Z, H^*)) _{ij} = A_{ij}(Q_{ii}-Q_{ij})-\beta\bar{A}_{ij}, \text{  for all }i,j=1,2,\ldots,n.
\end{align} {Since $\bar{Z}$ differs from $Z^*$ only at positions $(\bar{i}, \bar{j})$ and $(\bar{j}, \bar{i})$, and $\bar{Z}_{\bar{i}\bar{j}} - Z^*_{\bar{i}\bar{j}} = \bar{Z}_{\bar{j}\bar{i}} - Z^*_{\bar{j}\bar{i}} =1$, we know from \eqref{eq:gradient} that}
\begin{equation}
\label{eq:gap1}
\begin{aligned}
	f(\bar{Z},H^*)-f(Z^*,H^*) = &\nabla_Z f(Z, H^*) _{\bar{i}\bar{j}} (\bar{Z}_{\bar{i}\bar{j}}- {Z}^*_{\bar{i}\bar{j}}) + (\nabla_Z f(Z, H^*) )_{\bar{j}\bar{i}} (\bar{Z}_{\bar{j}\bar{i}}- {Z}^*_{\bar{j}\bar{i}})\\
	=& A_{\bar{i}\bar{j}}\left(Q_{\bar{i}\bar{i}}+Q_{\bar{j}\bar{j}}-Q_{\bar{i}\bar{j}}-Q_{\bar{j}\bar{i}}\right) - 2\beta \bar{A}_{\bar{i}\bar{j}}.
\end{aligned}
\end{equation}{Noting that $Q=H^*(H^*)\zz$ and $(H^*)\zz H^*=I_d$}, we get $Q_{\bar{i}\bar{i}}+Q_{\bar{j}\bar{j}}-Q_{\bar{i}\bar{j}}-Q_{\bar{j}\bar{i}}= \|H^*_{\bar{i},:}-H^*_{\bar{j},:}\|_2^2\leq 2$. Together with~\eqref{eq:gap1}, this fact implies that
\begin{equation*}
\begin{aligned}
	f(\bar{Z},H^*)-f(Z^*,H^*)
	\leq & 2 A_{\bar{i}\bar{j}} - 2\beta \bar{A}_{\bar{i}\bar{j}}= 2 A_{\bar{i}\bar{j}} - 2\beta p {A}_{\bar{i}\bar{j}} = 2A_{\bar{i}\bar{j}} (1-\beta p).
\end{aligned}
\end{equation*}
By the definitions of ${(\bar{i},\bar{j})}$ and $\epsilon$, we have $A_{\bar{i}\bar{j}}>{\epsilon}/{2}> 0$. Recalling that $\beta p>2$, we have
$$  2A_{\bar{i}\bar{j}} (1-\beta p) < -2A_{\bar{i}\bar{j}}<-\epsilon,$$ 
which then implies $f(\bar{Z},H^*)-f(Z^*,H^*)<-\epsilon$.
This contradicts the  {assumption} 
and completes the proof. 
\end{proof}

\begin{remark}
\label{rem:8}
\textit{We note that~\cite[Section 5]{kamvar2003spectral} proposes   incorporating the must-link constraints into $A$ by  }
\begin{equation}
\label{eq:Acon}
\widetilde{A}_{ij}:= 
\left\{
\begin{aligned}
	&1, \quad\text{  if } (i,j) \in \mathcal{J},\\
	&A_{ij}, \text{  otherwise}.
\end{aligned}\right.
\end{equation} 
\textit{It then performs spectral clustering using $\widetilde{A}$ as the similarity matrix.
Its approach is distinct from our construction in \eqref{eq:Aconstruct}, where we scale the original weights in $A$ using a parameter $p\ge 1$.
In addition, the method in \cite{kamvar2003spectral} does not present the theoretical guarantee that the must-link constraints can be satisfied in the final clustering results.
In contrast, as {demonstrated} in Theorem~\ref{thm:mustlink2}, our method ensures that $x_i$ and $x_j$ are always assigned to the same output cluster for all $({i},{j})\in \mathcal{J}$ under mild conditions.  
}
\end{remark}

\begin{remark}
	\label{rem:add}
	\revise{We emphasize that the condition $\beta p>2$ on the parameter $\beta$ is sufficient to ensure the satisfaction of the must-link constraints. However, this condition is not necessary, as will be demonstrated later through numerical experiments.}
\end{remark}

\section{{A} block coordinate descent algorithm for solving~\eqref{eq:model}}
\label{sec:3}
In this section, we introduce an algorithm specifically designed to solve problem~\eqref{eq:model} by leveraging its block structure. We prove that the algorithm converges to a blockwise 
$\epsilon$-minimizer (see Definition~\ref{def:mini}) within a finite number of iterations. Furthermore, we demonstrate that any iterate generated by the algorithm satisfies the must-link requirements.

\subsection{Algorithm framework}
We propose a block coordinate descent algorithm for solving problem~\eqref{eq:model}, which alternately updates the variables 
$Z$ and 
$H$.
Let $Z^{(0)}\in\snA \cap \{0,1\}^{n\times n}$ and $H^{(0)}\in\R^{n\times d}$ be the initial values. We now show how to update $(Z,H)$ at each iteration~$t\ge 0$.

Given  $H^{(t)}$, we first obtain $Z^{(t+1)}$ by solving a linear programming problem
\begin{equation}\label{eq:subZ}
\begin{aligned}
\min_Z \,\,&f(Z,H^{(t)})\\
\st\,\, &Z\in\snA \cap [0, 1]^{n\times n}.
\end{aligned}
\end{equation}
Since $f$ is linear with respect to $Z$, $Z^{(t+1)}$ defined as follows is a global solution to problem~\eqref{eq:subZ}:
\begin{equation}\label{eq:updateZ}
Z_{ij}^{(t+1)} :=
\begin{cases}
0,  &\text{  if~} G^{(t)}_{ij}> 0, \\
Z_{ij}^{(t)},   &\text{  if~} G^{(t)}_{ij}= 0, \\
1,   &\text{  if~} G^{(t)}_{ij}< 0,
\end{cases}
\end{equation}
where 
\begin{align}
	\label{eq:grag}
G^{(t)}:=\nabla_Z f(Z^{(t)}, H^{(t)} ) + \nabla_Z f(Z^{(t)}, H^{(t)})\zz.
\end{align}
By induction, 
the sequence $\{Z^{(t)}\}$ remains in the set $\snA \cap\{0, 1\}^{n\times n}$ throughout the iterations. This follows from the fact that   $Z^{(0)}\in \snA \cap\{0, 1\}^{n\times n}$ and $G^{(t)}\in \snA$, which can be checked according to the formulation of $\nabla_Z f$ given in Appendix \ref{appen:derive}.

Given $Z^{(t+1)}$, we obtain $H^{(t+1)}$ by minimizing $f(Z^{(t+1)},H)$ with respect to $H$,
		which can be formulated into
the optimization problem
\begin{equation}\label{eq:subH}
\begin{aligned}
\min_{H\zz   H=I_{d}} \,\,&\tr\left(H\zz \LL\left(A\circ Z^{(t+1)}\right)H\right).
\end{aligned}
\end{equation}
By Rayleigh-Ritz theorem,
the matrix $(u_1, \dots, u_{d})$ is a global solution to problem~\eqref{eq:subH},
where $\{u_1, \dots, u_{d}\}$ is
an orthonormal set, and $u_i$ is an eigenvector of $\LL(A\circ Z^{(t+1)})$ associated with $\lb_i(\LL(A\circ Z^{(t+1)}))$ for each $i=1,2,\ldots,d$. 
Instead of requiring a global solution,
our theory allows for an inexact computation of $H^{(t+1)}$, which satisfies
$f(Z^{(t+1)}, H^{(t+1)})\leq f(Z^{(t+1)}, H^{(t)})$, and
\begin{equation}
\label{eq:terminateH}
(H^{(t+1)})\zz H^{(t+1)} = I_d, \quad  0 \leq f(Z^{(t+1)}, H^{(t+1)})- \min_{H\zz   H=I_{d}} f(Z^{(t+1)},H) \leq \epsilon,
\end{equation}
where $\epsilon\geq 0$ is a prescribed precision.

The proposed  algorithm for solving problem~\eqref{eq:model} is summarized in Algorithm~\ref{alg:CP-BCD}.     

\begin{algorithm}[htbp]
\caption{{A} block coordinate descent algorithm for solving problem~\eqref{eq:model}}\label{alg:CP-BCD}
\begin{algorithmic}[1]
\State 
\textbf{Input:} $Z^{(0)}\in\snA \cap \{0,1\}^{n\times n}$, and an orthogonal matrix $H^{(0)}\in \R^{n\times d}$. Set $t:=0$. 

\For{$t=0,1,2,\ldots$}

\State  Calculate $Z^{(t+1)}$ by~\eqref{eq:updateZ}. 

\State If $f(Z^{(t)},H^{(t)})-f({Z}^{(t+1)},H^{(t)})  \leq {\epsilon}$, terminate the loop.

\State  Obtain $H^{(t+1)}$ by solving subproblem~\eqref{eq:subH} inexactly, so that~\eqref{eq:terminateH} holds.  

\EndFor

\State \textbf{Output}: $(Z^{(t)},H^{(t)})$.
\end{algorithmic}
\end{algorithm}

\subsection{Convergence analysis of Algorithm~\ref{alg:CP-BCD}}\label{sec:conver}
{We {show} that
	given any $\epsilon\ge0$, 
	Algorithm~\ref{alg:CP-BCD} outputs
	a	blockwise $\epsilon$-minimizer as defined below.  
\begin{definition}
	\label{def:mini}
	Let $\epsilon\ge 0$.
	We call $(Z^*,H^*)$ a {blockwise $\epsilon$-minimizer} of problem~\eqref{eq:model} if it holds that $Z^*\in\snA \cap \{0, 1\}^{n\times n}$,  $(H^*)\zz H^*=I_d$, and
	\begin{equation*} 
		\begin{aligned}
			&f\left(Z^*, H^*\right) -  \underset{Z \in \snA \cap[0,1]^{n \times n}}{\operatorname{min}} f\left(Z, H^*\right)\leq \epsilon, \quad f\left(Z^*, H^*\right) - \underset{H\zz H=I_{d}}{\operatorname{min}} f\left(Z^*, H\right)\leq \epsilon.
		\end{aligned}
	\end{equation*} 
\end{definition}
We note that any {blockwise $0$-minimizer} is a {blockwise minimizer} studied in~\cite{chen2019run}, and must be a KKT point.  
\begin{theorem}\label{thm:bcdconver}
Let $\epsilon\geq 0$. 
The following statements hold for Algorithm~\ref{alg:CP-BCD}.

\textnormal{(a)} ~For each $t$ before termination, either $f(Z^{(t)},H^{(t)})> f(Z^{(t+1)},H^{(t)})$ or $Z^{(t+1)}= Z^{(t)}$.

\textnormal{(b)} The algorithm terminates within finite number of iterations. Upon termination, it outputs a blockwise $\epsilon$-minimizer of problem~\eqref{eq:model}.
\end{theorem}
\begin{proof}
(a) Since $f(\cdot, H)$ is linear, by the definition $G^{(t)}$ in \eqref{eq:grag}, we have
\begin{equation}
\label{eq:fdiff}
f(Z^{(t)}, H^{(t)}) - f(Z^{(t+1)}, H^{(t)}) = \frac{1}{2} \left\langle G^{(t)}, Z^{(t)} - Z^{(t+1)} \right\rangle.
\end{equation}
According to~\eqref{eq:updateZ}, 
 $G^{(t)}_{ij}(Z^{(t)}_{ij} - Z^{(t+1)}_{ij}) >0$  for any $(i,j)$ such that $Z^{(t)}_{ij} - Z^{(t+1)}_{ij} \neq 0$. Therefore,~\eqref{eq:fdiff} ensures that $f(Z^{(t)}, H^{(t)}) - f(Z^{(t+1)}, H^{(t)}) >0$  if $Z^{(t+1)} \neq Z^{(t)}$.

(b)   {We consider two cases based on the value of $\epsilon$.}
We first consider the case with $\epsilon>0$. Since  
we have $f(Z^{(t)},H^{(t)})\geq f({Z}^{(t+1)},H^{(t)}) \geq f({Z}^{(t+1)},H^{(t+1)})$ for all $t$ before termination, the sequence $\{f(Z^{(t)},H^{(t)})\}$ is non-increasing. {Moreover, $\{f(Z^{(t)},H^{(t)})\}$ is bounded below, since $f$ is continuous and $\{(Z^{(t)},H^{(t)})\}$ is contained in the compact feasible set.}
{Therefore},  the stopping criterion $f(Z^{(t)},H^{(t)})-f({Z}^{(t+1)},H^{(t)})  \leq {\epsilon}$ must be met within finite number of iterations.
Suppose that Algorithm~\ref{alg:CP-BCD} terminates at the $\bar{t}$-th iteration.  
By the updating rule \eqref{eq:updateZ} of $Z$, it holds that
$$
f(Z^{(\bar{t})},H^{(\bar{t})}) - \min_{Z \in \snA \cap[0,1]^{n \times n}} f\left(Z, H^{(\bar{t})}\right) = f(Z^{(\bar{t})},H^{(\bar{t})}) -  f(Z^{(\bar{t}+1)},H^{(\bar{t})}) \leq \epsilon.
$$
Additionally, we have $ (H^{(\bar{t})})\zz{H}^{(\bar{t})}={I}_d$ and
$
f(Z^{(\bar{t})}, H^{(\bar{t})})-\min _{H^{\top} H=I_d} f(Z^{(\bar{t})}, H)\leq \epsilon
$
(see~\eqref{eq:terminateH}).
Thus, $(Z^{(\bar{t})},H^{(\bar{t})})$ is a  blockwise $\epsilon$-minimizer of problem~\eqref{eq:model}.

We next consider the case with $\epsilon=0$. {Due to the updating rules of $Z$ and $H$}, the sequence $\{f(Z^{(t)},H^{(t)})\}$ stays in the set $$\Omega:=\left\{\min_{H\zz   H=I_{d}}  f(Z,H)\mid  Z\in \snA  \cap  \{0,1\}^{n\times n}\right\},$$
which is  a finite set. In addition,
$\{f(Z^{(t)},H^{(t)})\}$ is non-increasing   {as mentioned in part (a)}, the sequence must eventually stabilize at some constant value after a finite number of iterations.  
{Once this occurs, the sequence $\{Z^{(t)}\}$ remains unchanged  for all subsequent iterations due to~(a).}
Suppose that the algorithm outputs $(Z^{(\bar{t})}, H^{(\bar{t})})$ when it terminates. It then holds that
\begin{equation}
\begin{aligned}
	&Z^{(\bar{t})}=Z^{(\bar{t}+1)} \in \underset{Z \in \snA \cap\{0,1\}^{n \times n}}{\operatorname{argmin}} f(Z, H^{(\bar{t})}), \quad H^{(\bar{t})} \in \underset{H\zz H=I_{d}}{\operatorname{argmin}} f(Z^{(\bar{t})}, H),
\end{aligned}
\end{equation}
which yields that $(Z^{(\bar{t})},H^{(\bar{t})})$ is a blockwise minimizer of problem~\eqref{eq:model}. 

{In both cases, the algorithm terminates after a finite number of iterations and outputs a blockwise $\epsilon$-minimizer of the problem. The proof is then completed.}
\end{proof}

{
	According to Theorem \ref{thm:bcdconver}(b), when $\beta p > 2$ and $\epsilon$ is small enough, the output point of Algorithm~\ref{alg:CP-BCD} satisfies the conditions stated in Theorem~\ref{thm:mustlink2}. Thus, the associated graph meets the must-link requirements. Indeed, as we will show in the following proposition, when $\beta p > 2$ and $\epsilon\ge 0$, all iterates of Algorithm~\ref{alg:CP-BCD} also adhere to the conditions specified in Theorem~\ref{thm:mustlink2}, and therefore meet the must-link requirements.
	\begin{proposition}
		\label{prop:alg}
		Suppose that $\beta p > 2$, $\epsilon\ge 0$, and the sequence $\{Z^{(t)}, H^{(t)}\}_{t\ge 0}$ is generated by Algorithm \ref{alg:CP-BCD}. Then $Z^{(t)}_{{i}{j}}= 1$ for all $({i},{j})\in \mathcal{J}$ and $t\ge 1$. 
	\end{proposition}
	
	\begin{proof}
		Fix an iteration counter $t\ge 0$. Since $Z^{(t+1)}$ given in \eqref{eq:updateZ} is a global solution to problem~\eqref{eq:subZ},
		we know that $$Z^{(t+1)}\in\snA \cap \{0, 1\}^{n\times n}, \text{  and } f\left(Z^{(t+1)}, H^{(t)}\right) -  \underset{Z \in \snA \cap[0,1]^{n \times n}}{\operatorname{min}} f\left(Z, H^{(t)}\right)=0.$$ In addition,  
		we have $(H^{(t)})\zz H^{(t)}=I_d$ by the updating rule of $H$. Thus,   it follows from Theorem \ref{thm:mustlink2} that $Z^{(t+1)}_{{i}{j}}= 1$ for all $({i},{j})\in \mathcal{J}$.  
	\end{proof}
}




\section{Numerical experiments}
\label{sec:4}

In this section, we {demonstrate} the numerical performance of COSSC compared with {several} {widely-used} approaches.
All experiments are conducted by MATLAB R2017a and run on a MacBook Pro (13-inch, 2018) with 2.3 GHz Intel Core i5 and 8 GB of RAM.

We show the default settings of our experiments in Subsection~\ref{sec:default}. Our numerical tests are based on two types of data. 
In Subsection~\ref{sec:graph}, we illustrate the numerical results of COSSC on six synthetic datasets, each represented as a graph; this section focuses on the performance of COSSC with difference settings of $\beta$ and $p$, which  helps us set these two parameters for further experiments.
In Subsection~\ref{sec:docu}, we test the effectiveness of COSSC on a real-world large-scale dataset from the TDT2 test set; specifically, we compare COSSC with some other methods under different values of $d$ and varying numbers of must-link constraints; we also  compare the CPU time across different methods.


\subsection{Default settings}\label{sec:default}
Given a data matrix $X$, we take the matrix 
$A=\operatorname{Diag}(\widetilde{A} e_n)^{-1} \widetilde{A}$, where
$\widetilde{A}$ is the similarity matrix of the $k_n$-nearest neighbor graph \cite{brito1997connectivity},  with $k_n=\lceil \log(n)\rceil$ as suggested in~\cite{von2007tutorial}. Specifically, we let $\widetilde{A}_{ij}= \exp(-{||x_i-x_j||^2
_2}{\sigma_i^{-1}\sigma_j^{-1}})$ if $v_i$ is among the $k_n$-nearest neighbors of $v_j$ and $v_j$ is among the $k_n$-nearest neighbors of $v_i$; we let $\widetilde{A}_{ij}=0$ otherwise. Here, $\sigma_i$ represents the distance between $x_i$ and its $7$th neighbor as suggested in~\cite{von2007tutorial, zelnik2005self}. 

{Unless otherwise specified,} we set $p=10$  in~\eqref{eq:Aconstruct} and $\beta={(d-1)}/{n}$ in \eqref{eq:model}, which are suggested by the results that will be presented in Section \ref{sec:graph}.
{Although this} setting {may not} satisfy \revise{the condition $\beta p>2$,} as required by Theorem~\ref{thm:mustlink2},
numerical results in Subsection~\ref{sec:442} \revise{demonstrate} that 
the must-link constraints are \revise{still satisfied} under this default setting.
{In Algorithm~\ref{alg:CP-BCD},
we solve~\eqref{eq:subH} inexactly by {invoking} the ``eigs'' function in MATLAB, {which essentially calls the built-in ARPACK package.}
Besides, we set 
$\epsilon = 10^{-3}$ in Algorithm~\ref{alg:CP-BCD}. In addition to the stopping criterion in Line 4 of the algorithm, we terminate the calculation whenever
 $t=500$. We initialize Algorithm~\ref{alg:CP-BCD} with $
Z^{(0)}=  \sign(A).
$}

We 
\revise{compare} COSSC with the following existing  approaches to clustering, all of which require the precise number of clusters $k$ as an input, as mentioned in Section \ref{sec:intro}. 
\begin{enumerate}
\item Variants of $k$-means. We select  constrained $k$-means
and heuristic $k$-means~\cite{bair2013semi,miller2009geographic}. The code is downloaded from MathWorks File Exchange\footnote{https://www.mathworks.com/matlabcentral/fileexchange/117355-constrained-k-means}.
\item {Semi-supervised spectral clustering. We compare Algorithm~\ref{alg:CP-BCD} with SCA\footnote{In cases when there is no supervisory information, the steps of SCA are the same as those of standard spectral clustering, as introduced in Section 1.}, a semi-supervised spectral clustering method introduced in~\cite[Section 5]{kamvar2003spectral}.}
To apply SCA to the given dataset, we 
call  the ``spectralcluster'' function from the ``Statistics and Machine Learning Toolbox'' in MATLAB as follows: ``spectralcluster($\bar{A}$, $k$, `DISTANCE', `PRECOMPUTED')'', with $\bar{A}$ defined in~\eqref{eq:Acon}. 

\item Variants of NMF.  We choose Symmetric NMF (SymNMF)~\cite{kuang2012symmetric} and Semi-supervised NMF (SemiNMF)~\cite{yang2014unified} for the comparison. 
The code is downloaded from GitHub\footnote{https://github.com/dakuang/symnmf, https://github.com/zhanghuijun-hello/community-detection-code}.
\end{enumerate}
We run all the code mentioned above under the default settings.

Finally, 
{we introduce the} performance indicators. 
We {denote} the CPU time in seconds {as} ``time''. 
{To evaluate the clustering performance,} we use accuracy (ACC, the percentage of correctly clustered items), and the normalized mutual information~\cite{Estevez2009nmi} (NMI, a measure of the similarity between the ideal clusters and the generated clusters). 
Both of this indicators range from 0 to 1, with a larger NMI or ACC indicating a better clustering result (see more discussions in~\cite[Appendix B]{lancichinetti2009detecting}). 
{In addition, }{we use RMV to represent the Ratio of  Must-link {constraints} that are Violated by the output clusters, namely, RMV$ = {|\mathcal{J}_{{Z}^*}|}/{|\mathcal{J}|}$, where 
$
\mathcal{J}_{{Z}^*}: =\{(i,j) \mid Z^*_{ij}=0, (i,j)\in \mathcal{J}\}
$.}

\subsection{Numerical results on synthetic datasets}
\label{sec:graph}

In this subsection, we {study} the performance of COSSC in solving clustering problems with synthetic datasets, focusing on how different choices of
$\beta$ and $p$  influence the behavior of COSSC.
Specifically,
we {select} several synthetic datasets shown in Figure~\ref{pic:8graph}, where the ideal clusters are indicated by colors, with the ideal number $k^*$ of clusters being {3, 3, 3, 3, 2, and~5, respectively}. 
\begin{figure}[htbp!]
\centering
\subfigure[\tiny Dataset 1: three circles]{\includegraphics[height=25mm]{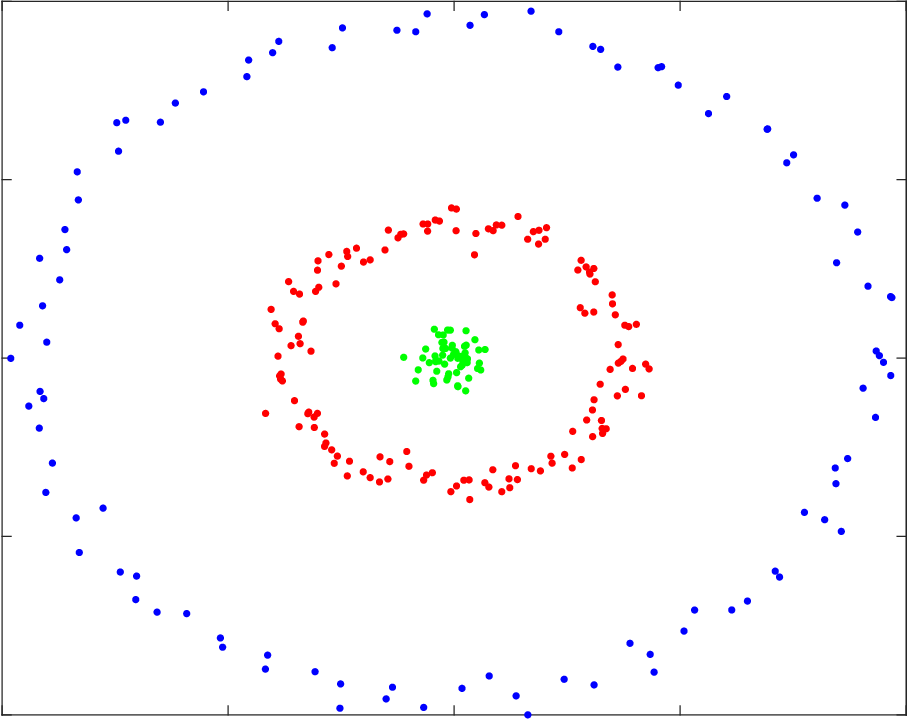}}\quad\quad\quad
\subfigure[\tiny Dataset 2: smile faces]{\includegraphics[height=25mm]{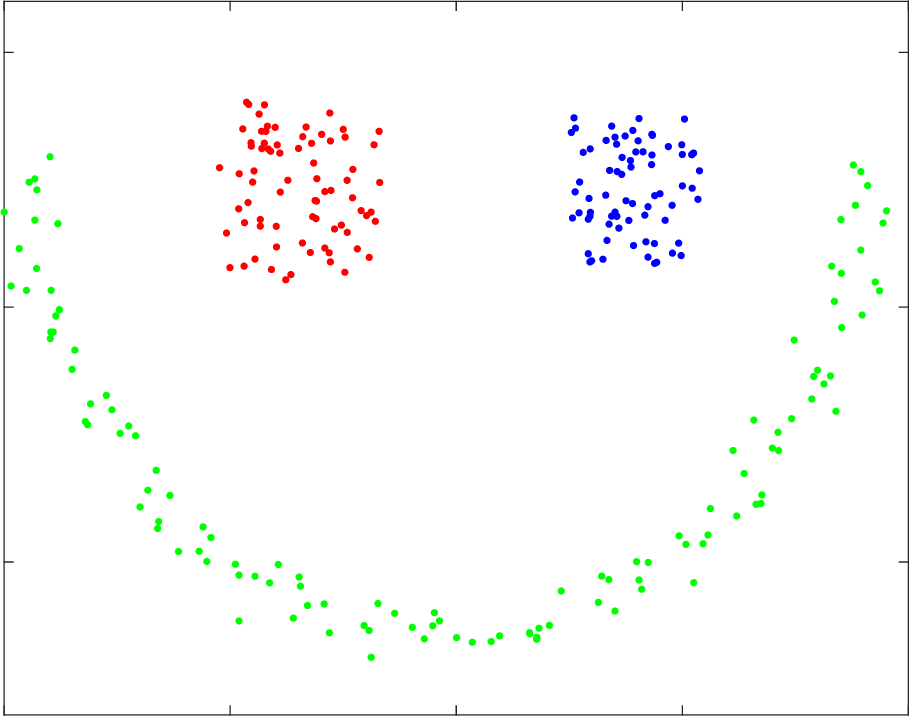}}\quad\quad\quad
\subfigure[\tiny Dataset 3: three parts]{\includegraphics[height=25mm]{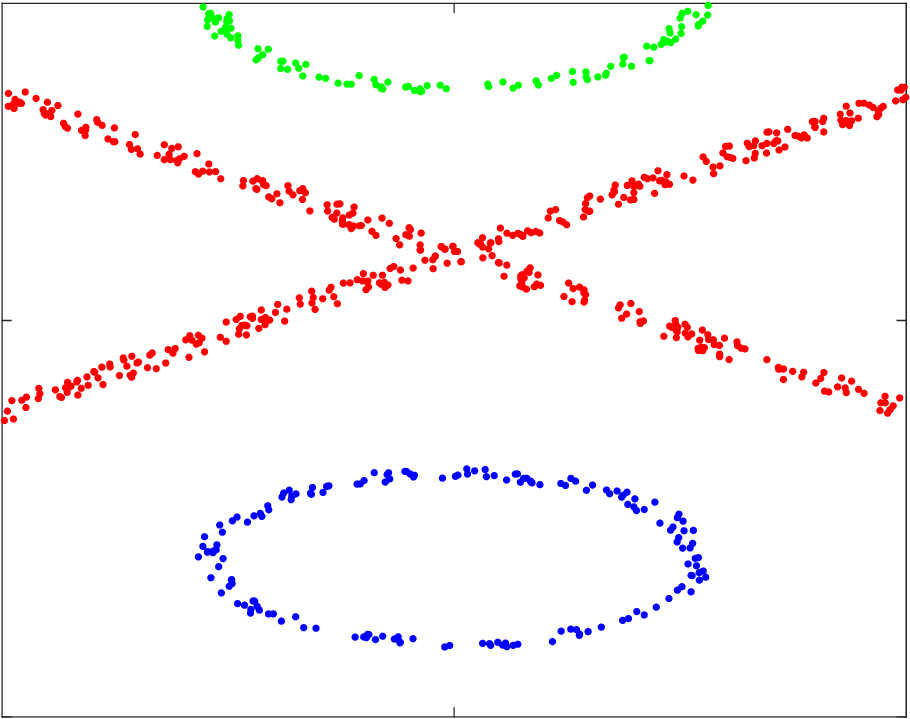}}\\
\subfigure[\tiny Dataset 4: two blocks in a circle]{\includegraphics[height=25mm]{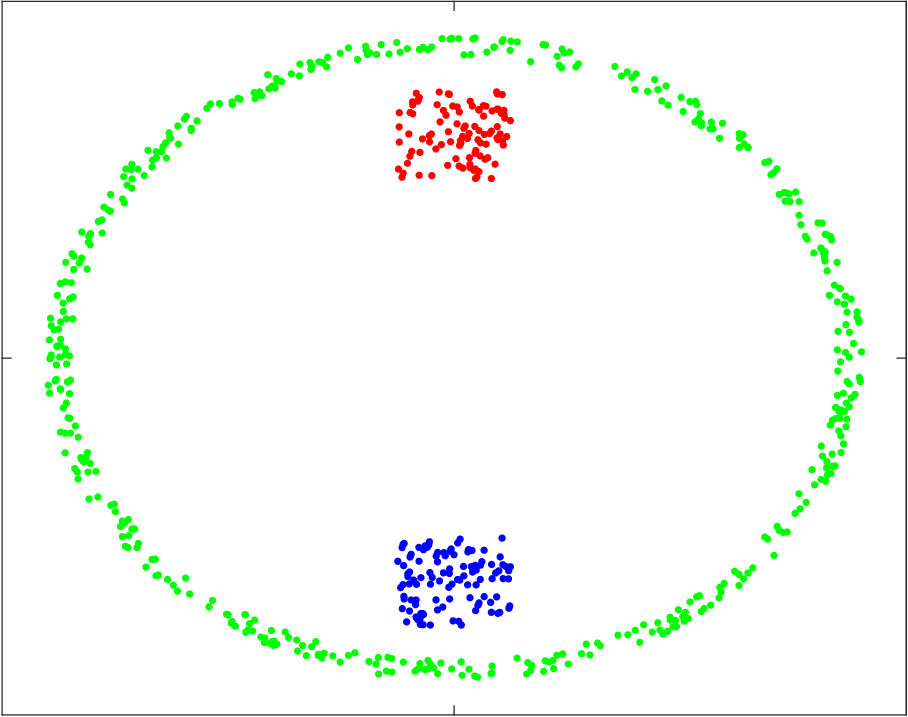}}\quad\quad\quad
\subfigure[\tiny Dataset 5: two moons]{\includegraphics[height=25mm]{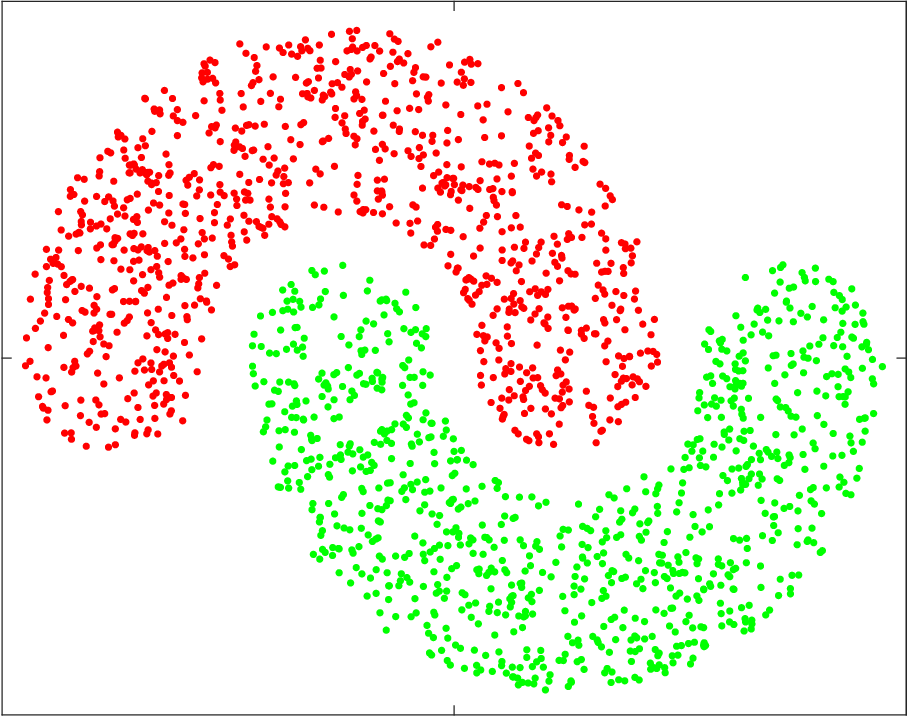}}\quad\quad\quad
\subfigure[\tiny Dataset 6: four blocks with noises]{\includegraphics[height=25mm]{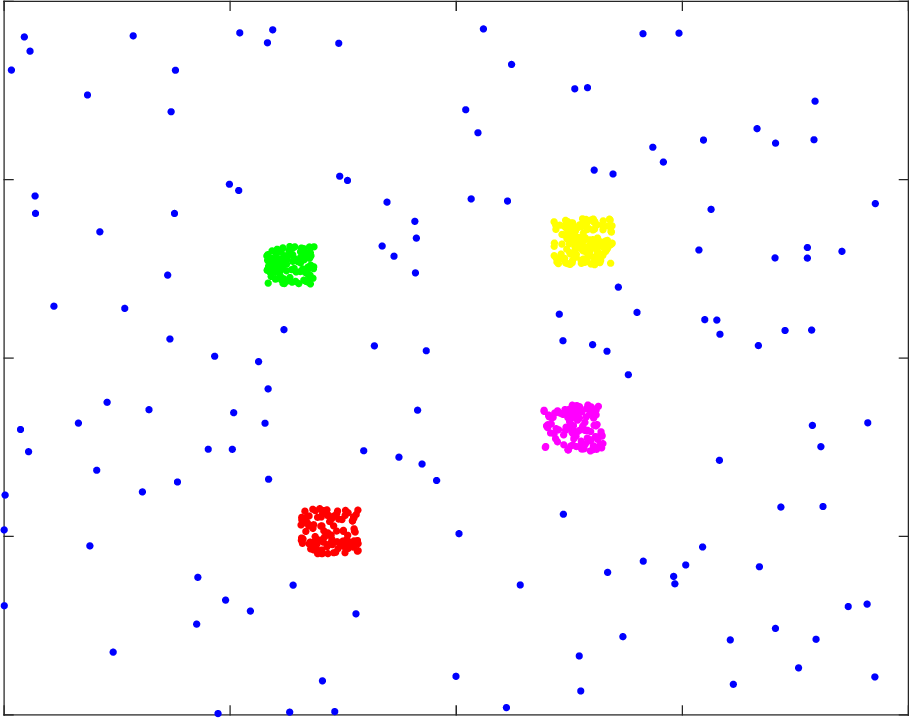}}
\caption{Synthetic graph datasets}\label{pic:8graph}
\end{figure}

\subsubsection{Choice of the parameter $\beta$}\label{sec:beta}

In this subsection, we test COSSC under different choices of~$\beta$ ranging in 
$\{1/10n, 1/n, (d-1)/n, 0.1\}$, while ensuring $d\geq k^*$. 
These four choices for $\beta$ are selected to balance the two terms in \eqref{eq:modelORIGIN}. Specifically, $1/10n$ and $1/n$ provide a scale based on the number of data points, $(d-1)/n$ takes into account the predefined $d$, and $0.1$ offers a fixed choice.
Then, we investigate how these choices of~$\beta$ affect ACC with varying input $d$. In this task, ACC=1 means that all the clusters are correctly separated.  The results are shown in Figure~\ref{fig:beta}, where several lines overlap at ACC$=1$, making it difficult to visually distinguish them in the graph.
The $x$-axis of each figure represents the value of $d$, and the $y$-axis represents the ACC. 
From these results, we observe that $\beta=(d-1)/n$ {is the best choice among all.  
Therefore, we take $\beta=(d-1)/n$ as the default setting in the following numerical experiments.}

\begin{figure}[htbp!]
\centering
\subfigure[Dataset 1]{\includegraphics[height=33mm]{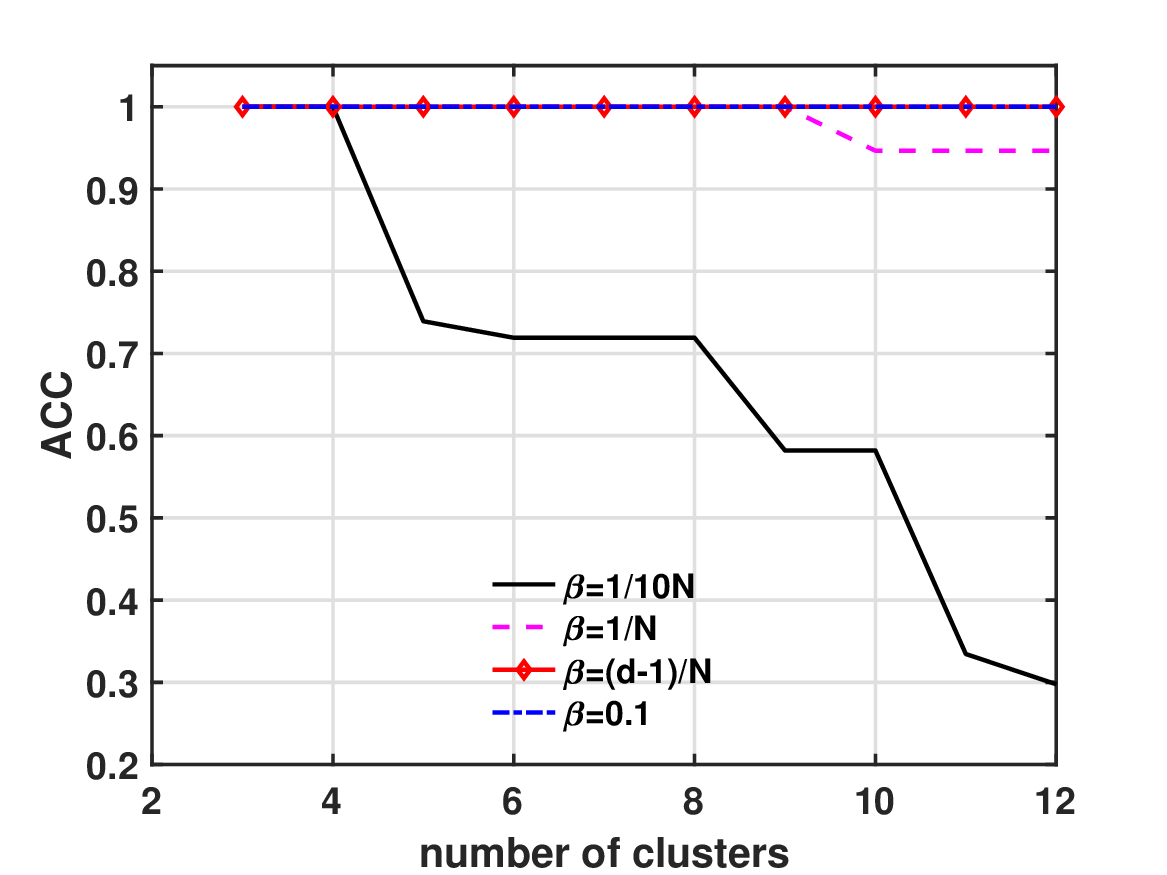}}\quad
\subfigure[Dataset 2]{\includegraphics[height=33mm]{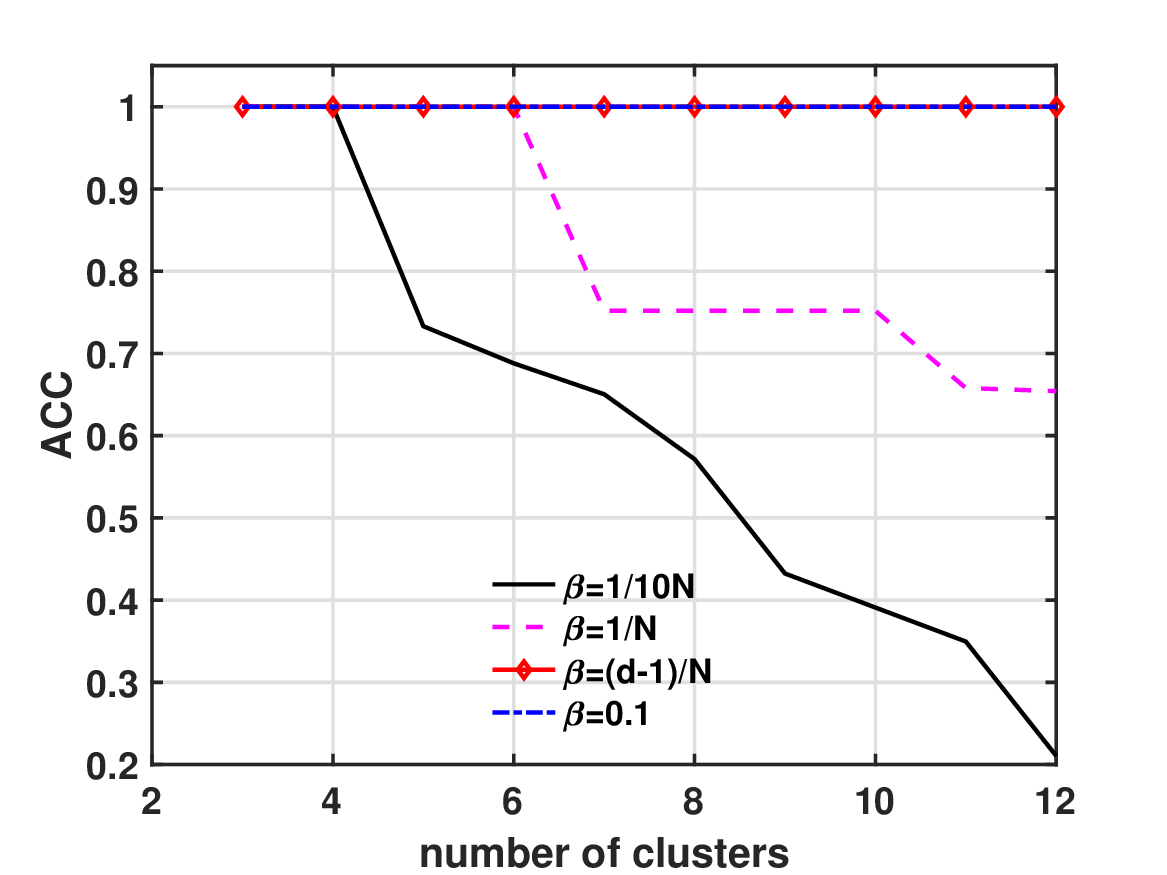}}\quad
\subfigure[Dataset 3]{\includegraphics[height=33mm]{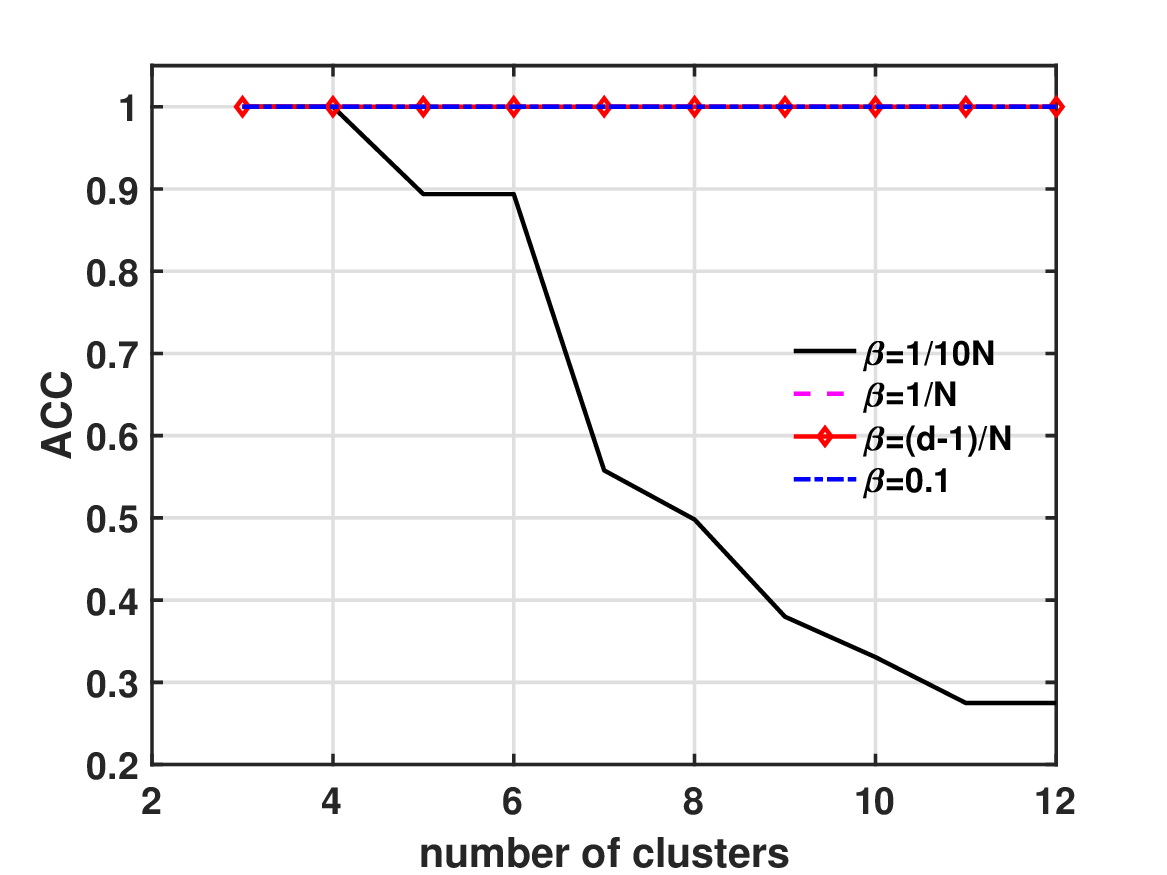}}\\
\subfigure[Dataset 4]{\includegraphics[height=33mm]{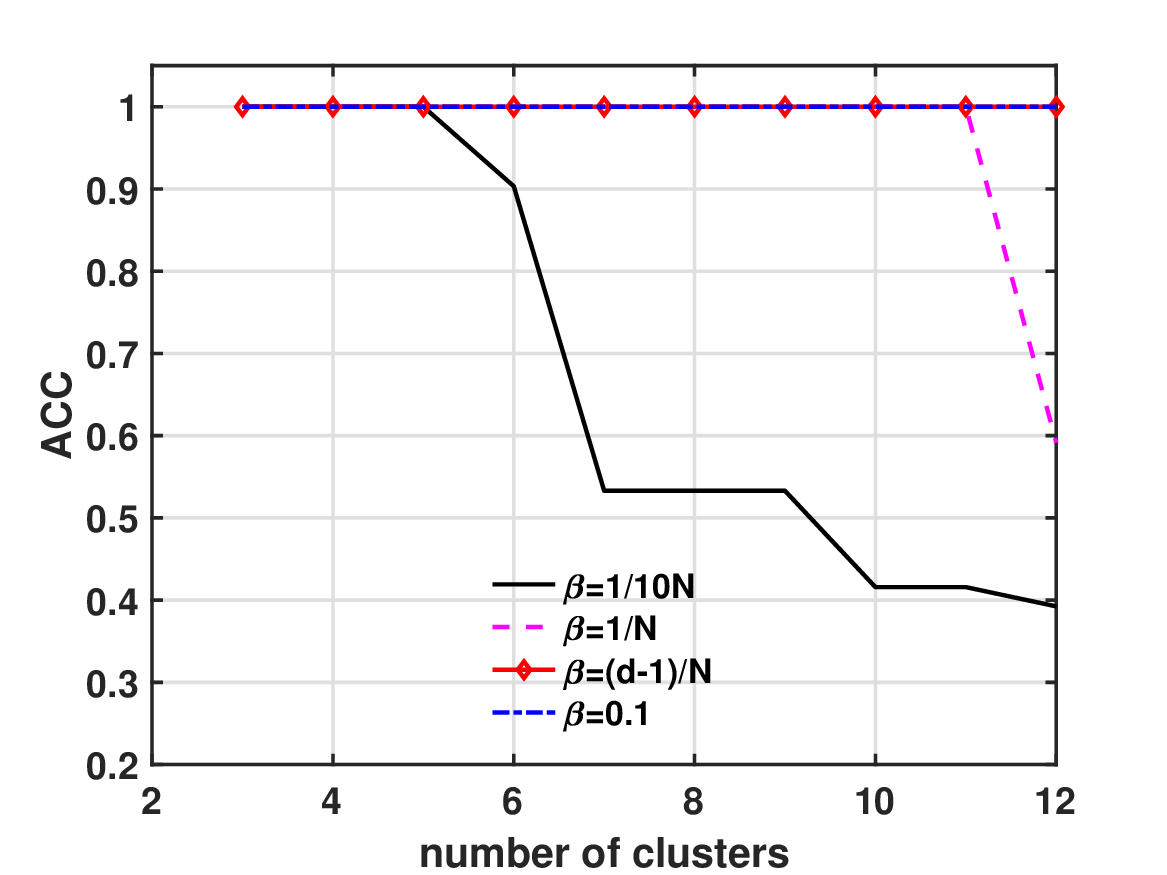}}\quad
\subfigure[Dataset 5]{\includegraphics[height=33mm]{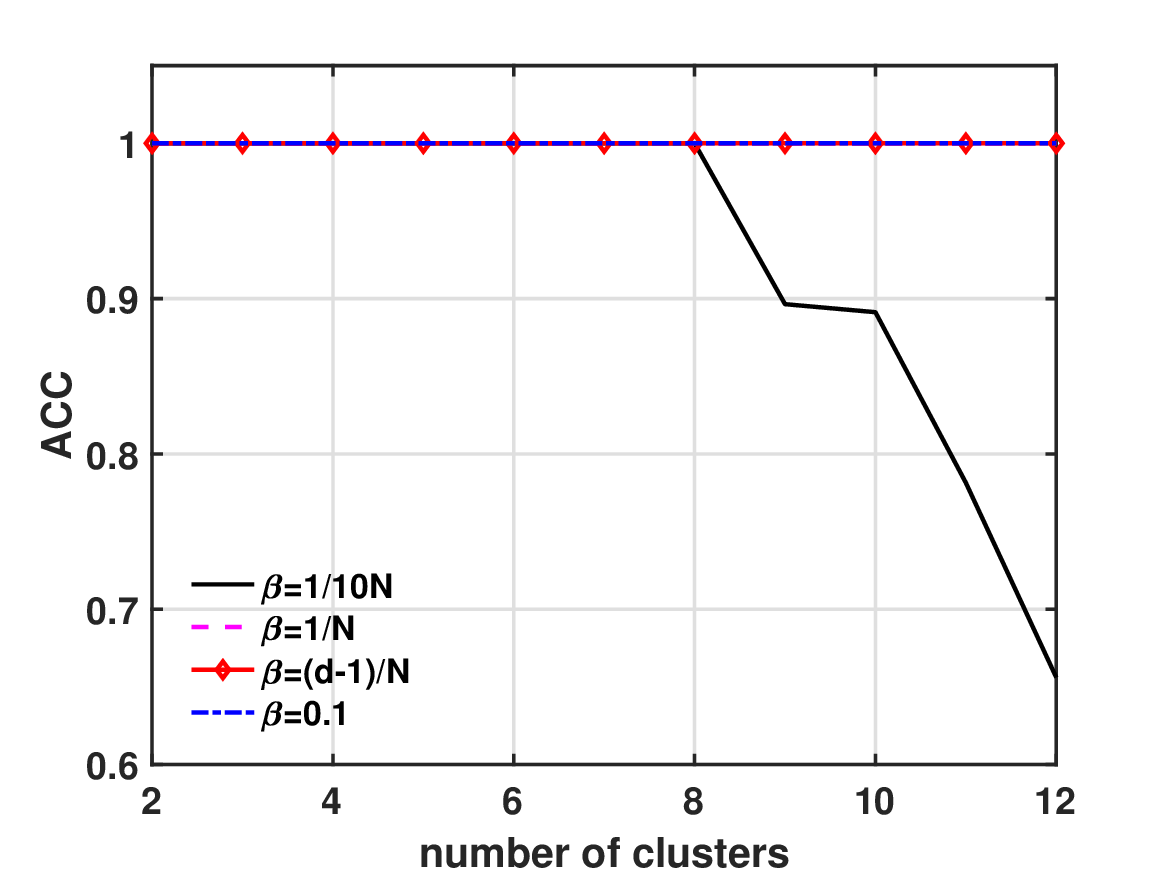}}\quad
\subfigure[Dataset 6]{\includegraphics[height=33mm]{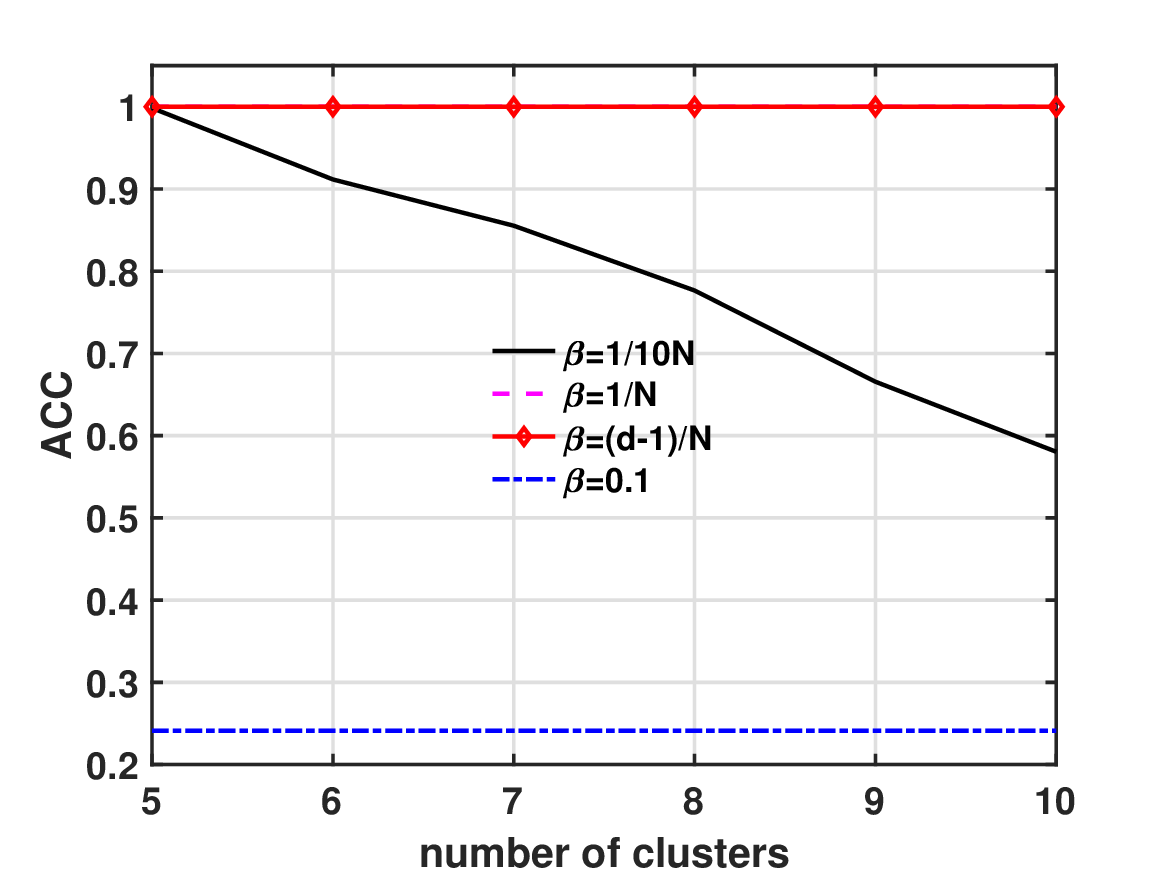}}\\
\caption{{A} comparison of ACC of COSSC with different $\beta$ and $d\geq k^*$, where several lines overlap at ACC$=1$.}\label{fig:beta}
\end{figure}

\subsubsection{\revise{Must-link constraints and choice of the parameter $p$}}

{{In this subsection, we illustrate how the parameters $p$ and $\beta$ affect the output clusters, in terms of satisfying the must-link constraints.}

First, we study how $p$ affects the performance of COSSC with the default $\beta={(d-1)}/{n}$ as suggested by the experiments in Section \ref{sec:beta}.
{We choose the dataset represented by Figure~\ref{pic:8graph}(f), impose several must-link constraints on certain edges, {and color them black.} The resulting graph is shown in  
Figure~\ref{pic:p}(a).}
Note that some must-link constraints are inconsistent with the natural (ideal) clusters. 
Under these constraints, we compare the performance of COSSC with $d=5$ and $p\in\{1,2,5\}$, and show the corresponding output clusters   in Figure~\ref{pic:p}(b)--(d).
From the results, we observe that no must-link constraint is satisfied {when} $p=1$,
one must-link constraint is satisfied {when} $p=2$, and all the must-link constraints are met {when $p= 5$.  
{This indicates that increasing $p$ helps satisfy the must-link constraints,}
which is guaranteed by Theorem~\ref{thm:mustlink2}.}  Based on these results, we take $p=10$ as our default setting.

{{Next, for the six datasets in Figure~\ref{pic:8graph}, we visualize the relationship between parameters $p, \beta$ and the RMV by comprehensive numerical experiments.}
Specifically, we randomly select $25\%$ of edges from these datasets\footnote{{We use the MATLAB function ``randperm''  to generate a random perturbation of the edge set $\{(1,2), (1,3), \ldots, (1,n), \ldots, (n-1,n)\}$. We then choose the top 25$\%$ of the  perturbation.}},  
put must-link constraints on them, and set $d=3,3,3,3,2,5$ for COSSC corresponding to these datasets, respectively. We choose $\beta$ from $\{0.001,0.002,0.003,\ldots,0.1\}$ and $p$ from $\{1.1,1.2,1.3,\ldots,10\}$, leading to 9000 distinct combinations  of $\beta$ and $p$. We apply COSSC with such combinations to these datasets, and present the numerical results in Figure~\ref{pic:p23}. 
In this figure, the shaded areas with different colors represent the values of $(\beta,p)$ for which the output clusters generated by COSSC meet a certain percentage of must-link constraints.  
Notably, the must-link constraints are satisfied across a wide range of $(\beta,p)$ values, beyond the one specified in Theorem~\ref{thm:mustlink2}.}

\begin{figure}[htbp]
\centering
\subfigure[Graph with three must-link constraints]{\includegraphics[height=27mm]{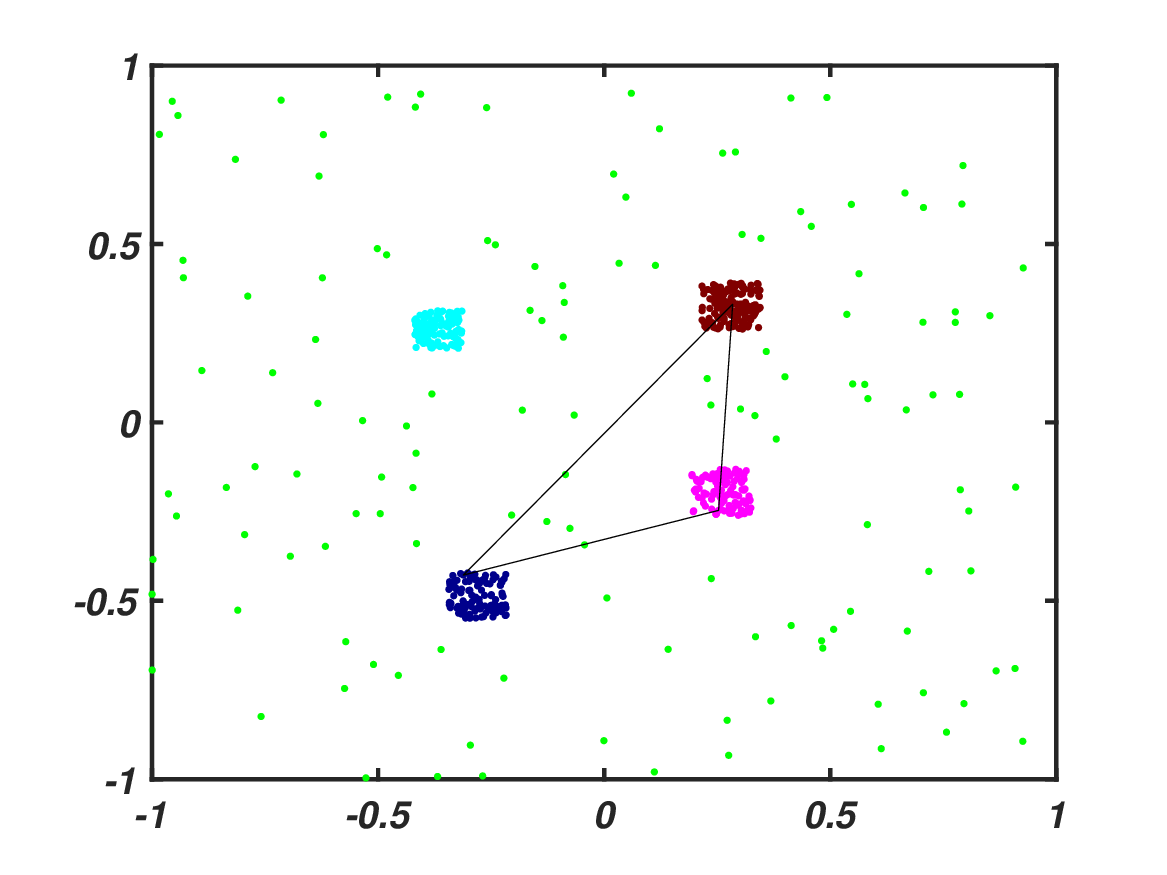}}
\subfigure[ $p=1$]{\includegraphics[height=27mm]{fig/gound6}}
\subfigure[ $p=2$]{\includegraphics[height=27mm]{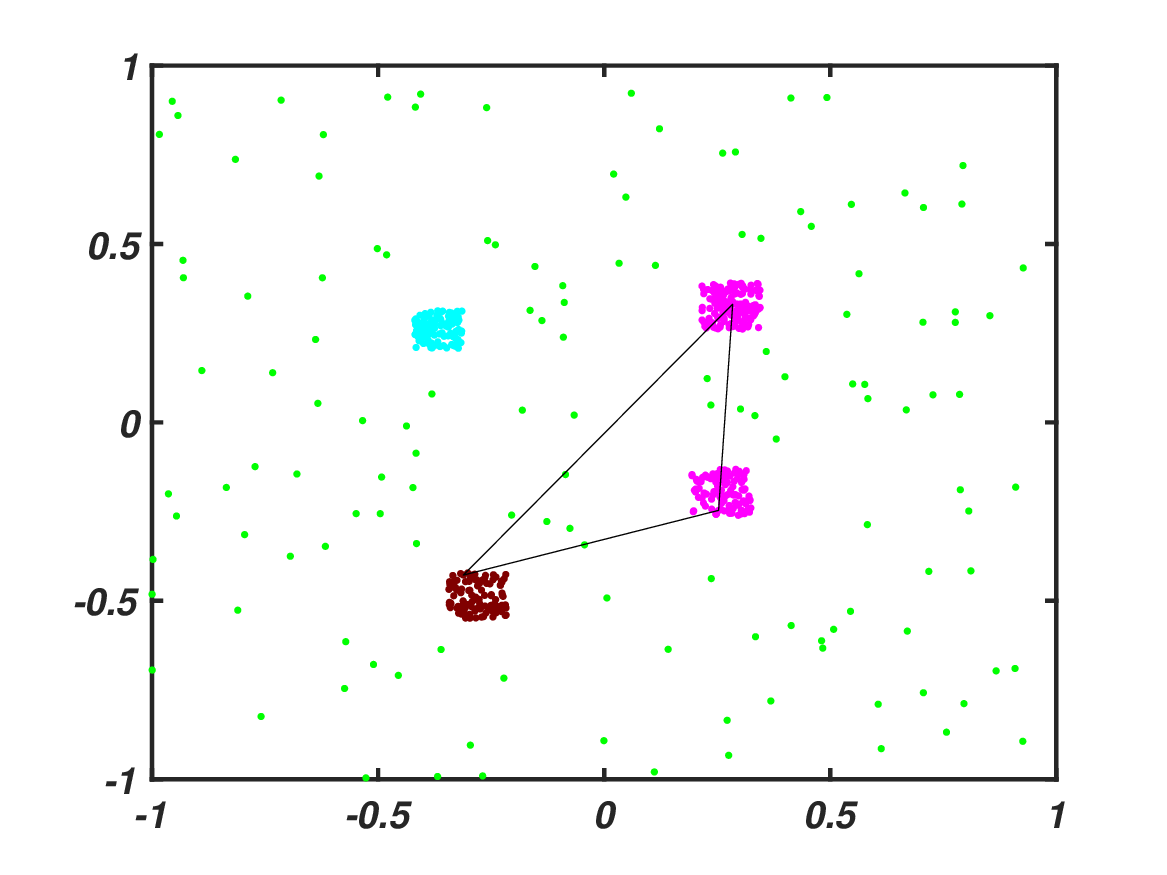}}
\subfigure[ $p=5$]{\includegraphics[height=27mm]{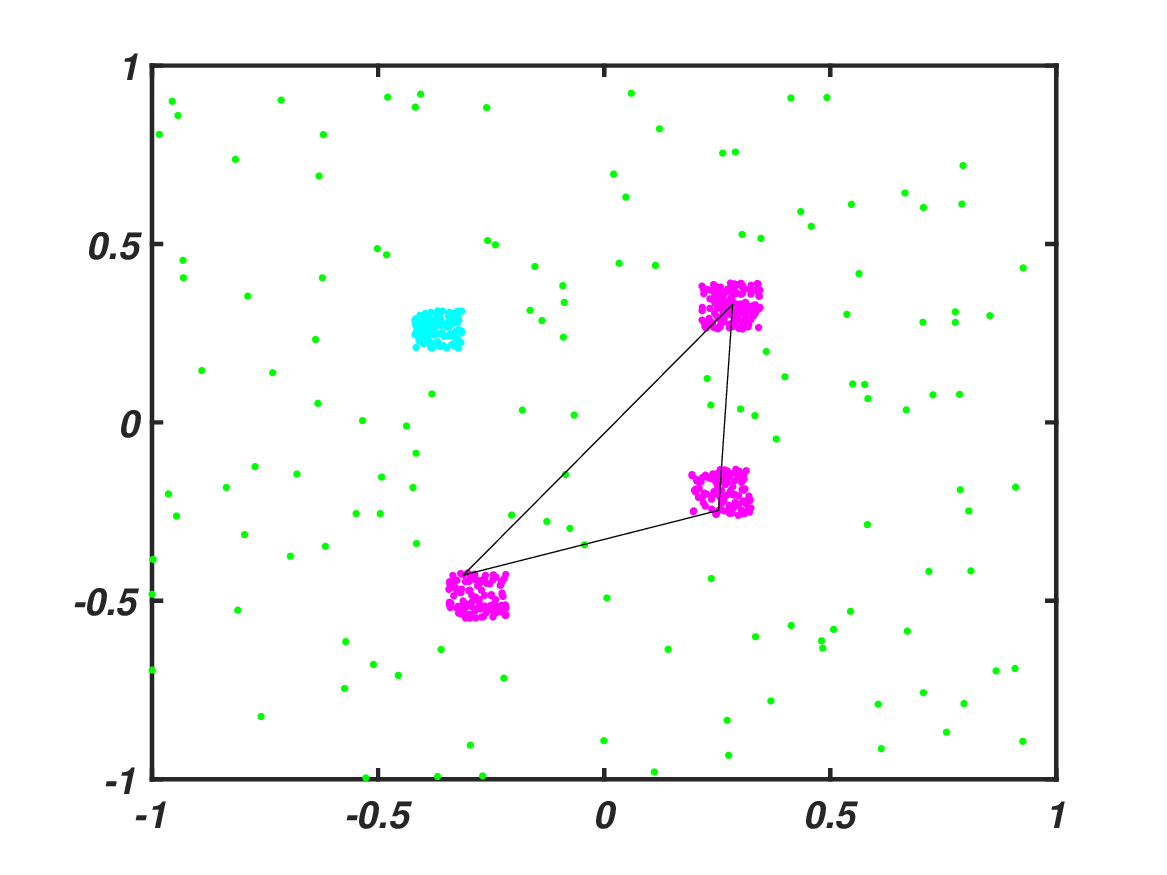}}
\caption{{A} comparison of the output clusters of COSSC with different $p$.}\label{pic:p}
\end{figure}

\begin{figure}[htbp]
\centering
\subfigure[ Dataset 1]{\includegraphics[height=33mm]{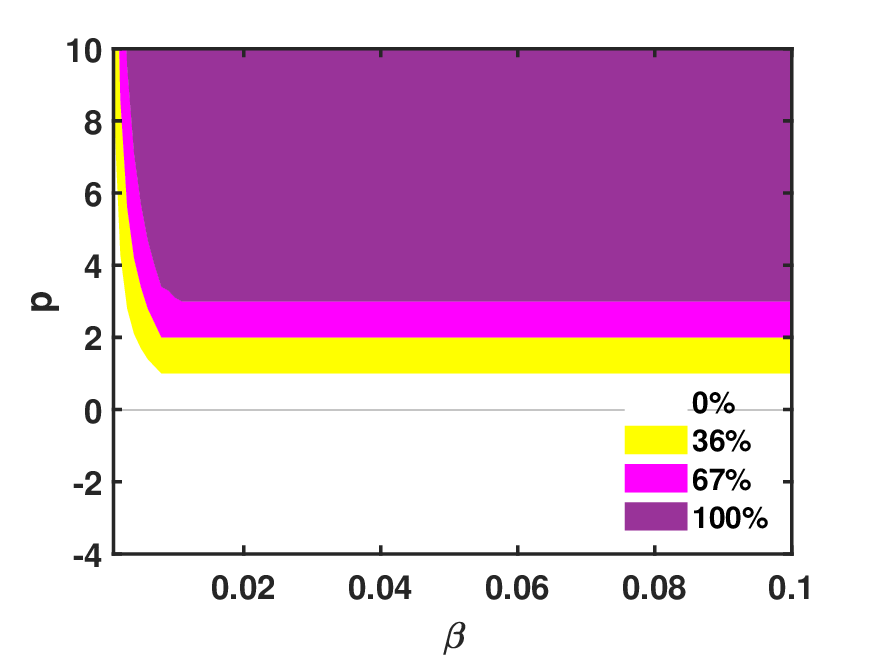}}\quad
\subfigure[  Dataset 2]{\includegraphics[height=33mm]{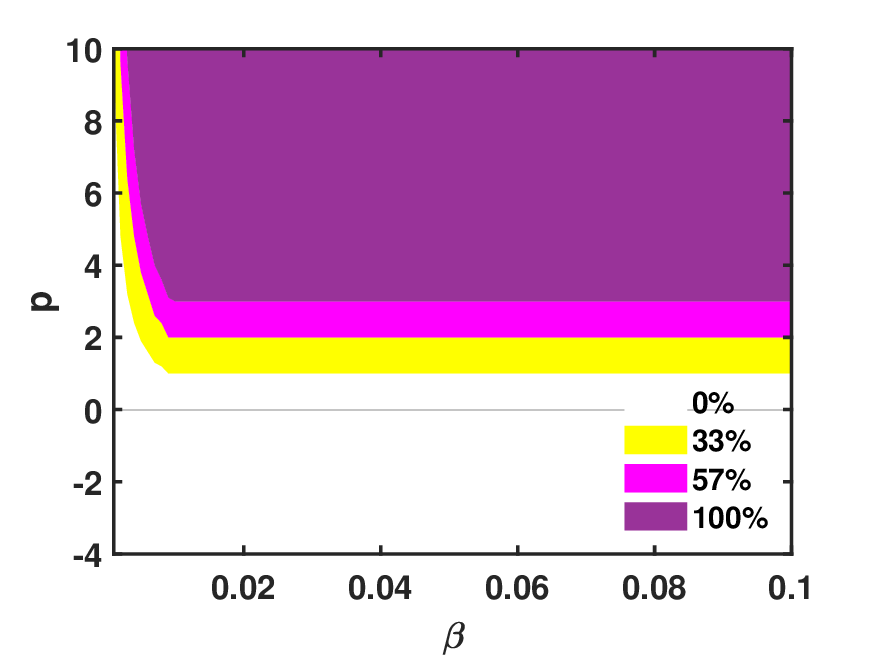}}\quad
\subfigure[ Dataset 3]{\includegraphics[height=33mm]{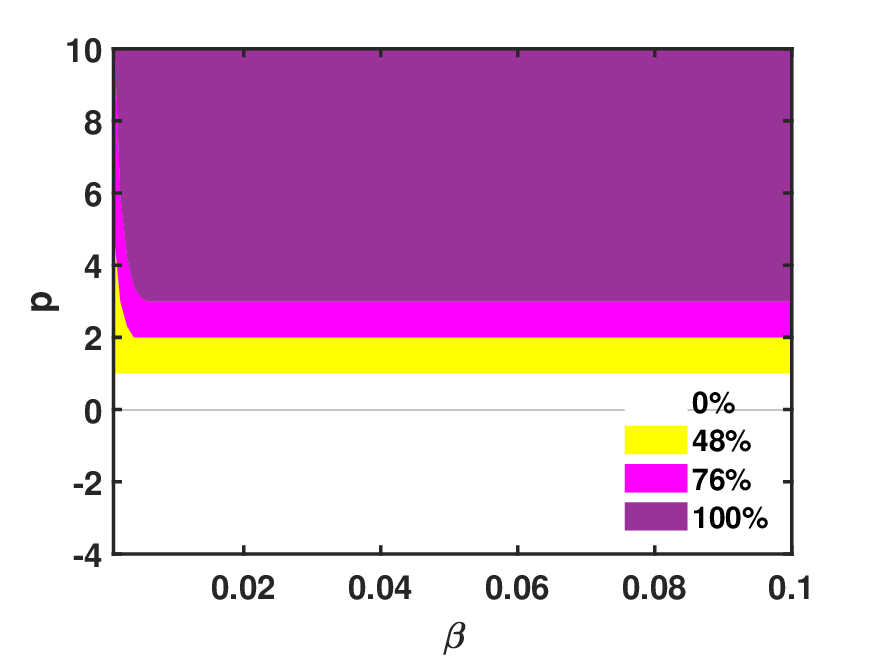}}\\
\subfigure[ Dataset 4]{\includegraphics[height=33mm]{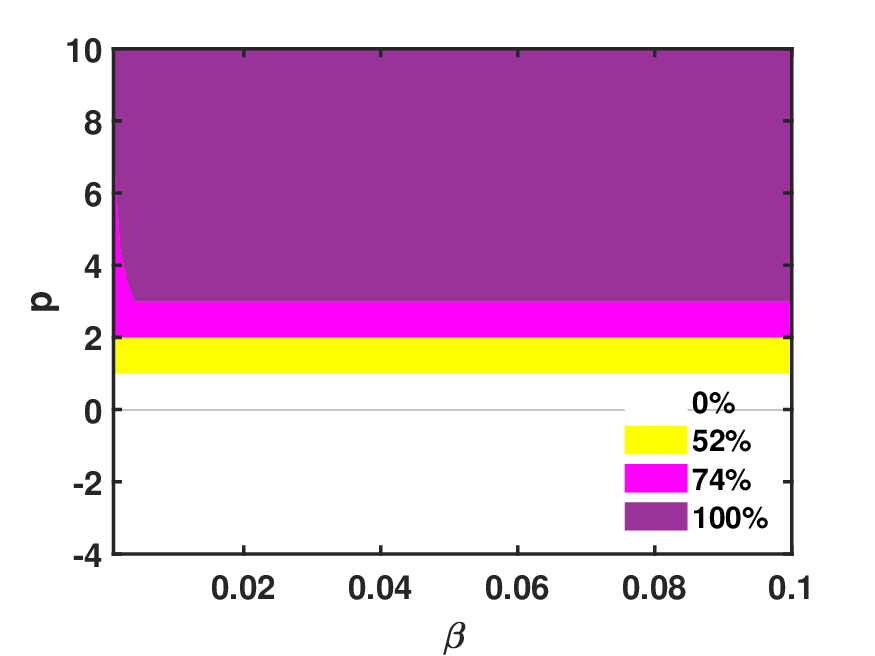}}\quad
\subfigure[  Dataset 5]{\includegraphics[height=33mm]{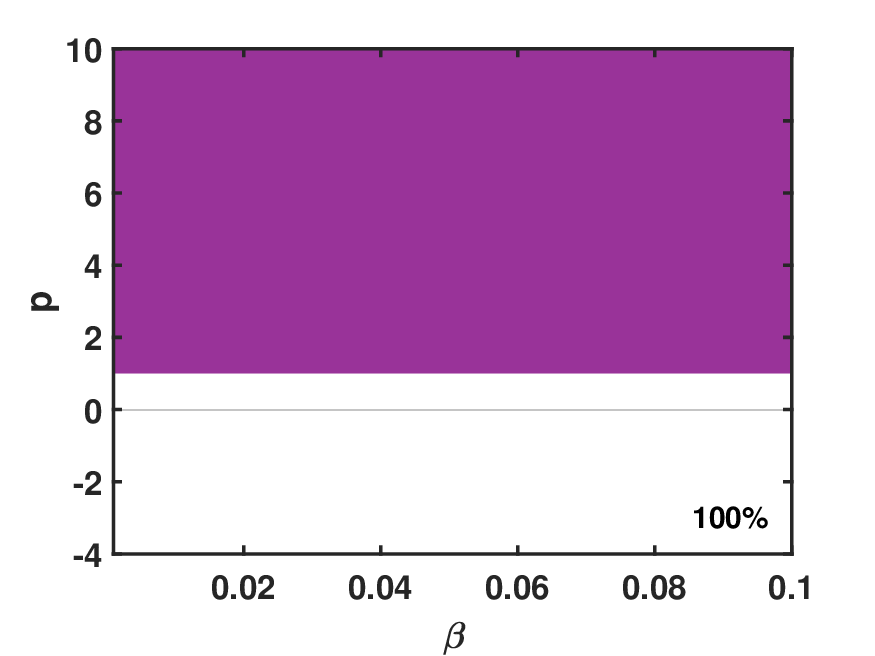}}\quad
\subfigure[ Dataset 6]{\includegraphics[height=33mm]{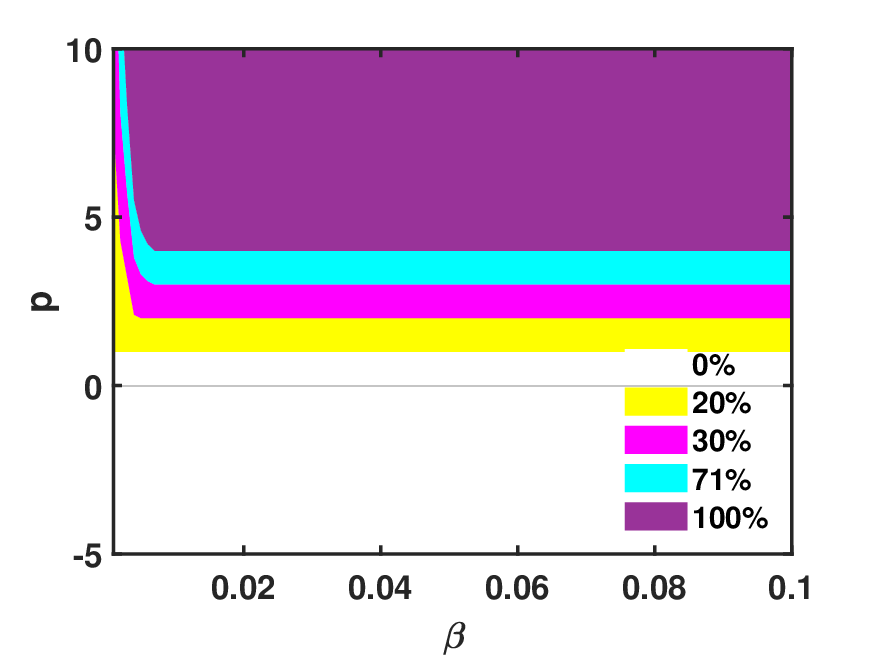}}
\caption{Ratio of the must-link {constraints} satisfied by output clusters on Figure \ref{pic:8graph} with different $\beta$ and $p$.}\label{pic:p23}
\end{figure}

\subsubsection{Numerical performance with different input number of clusters}\label{sec:graphdiffn}
{In this subsection, we show} that the numerical performance of COSSC is relatively insensitive to the choice of $d$ as long as it is not too large.  
As a comparison, we also test SCA with different input~$k$.
The numerical results are displayed in Figure~\ref{pic:9acc}. 
In the first two columns of this figure, we show that SCA outputs the ideal clusters when the input number of clusters $k$ equals the ideal number $k^*$ of clusters. However, it fails to produce the ideal clusters when $k=k^*+1$.  
{In contrast, as shown in the third column of Figure~\ref{pic:9acc}, COSSC consistently outputs the ideal clusters across a wide range of $d$ values. 
}

\begin{figure}[htbp]
\centering
\subfigure[\tiny Dataset 1, SCA, $k=k^*$=3]{\includegraphics[height=30mm]{fig/graph1_sc_right}}\quad\quad\quad\quad
\subfigure[\tiny Dataset 1, SCA, $k$=4]{\includegraphics[height=30mm]{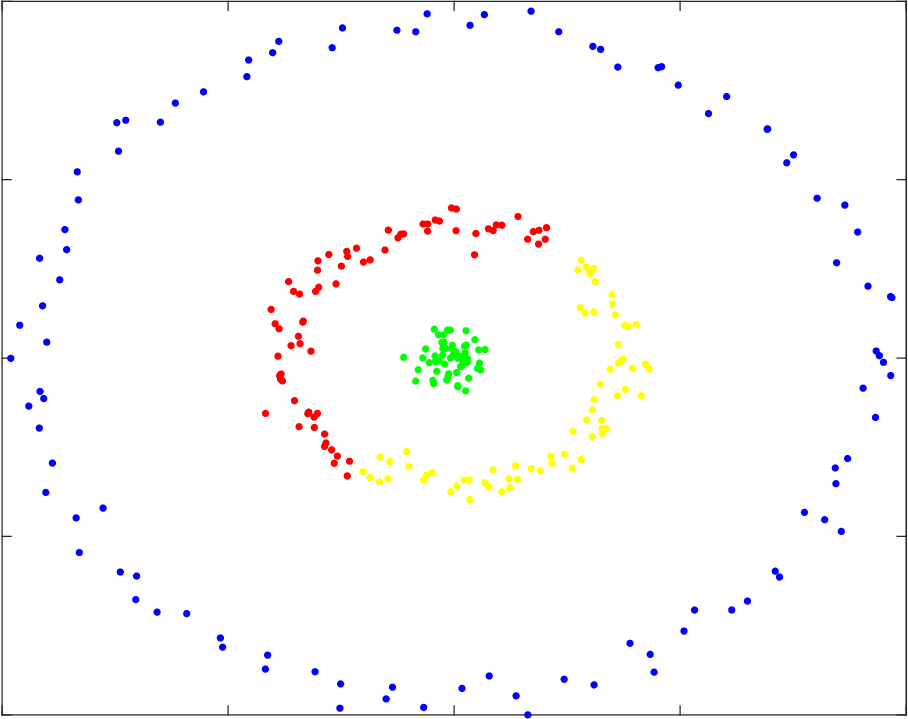}}\quad\quad\quad\quad
\subfigure[\tiny Dataset 1, COSSC, $d$=3,$\ldots$,12]{\includegraphics[height=30mm]{fig/graph1_sc_right}}\\
\vspace{-3mm}
\subfigure[\tiny Dataset 2, SCA, $k=k^*$=3]{\includegraphics[height=30mm]{fig/graph2_sc_right}}\quad\quad\quad\quad
\subfigure[\tiny Dataset 2, SCA, $k$=4]{\includegraphics[height=30mm]{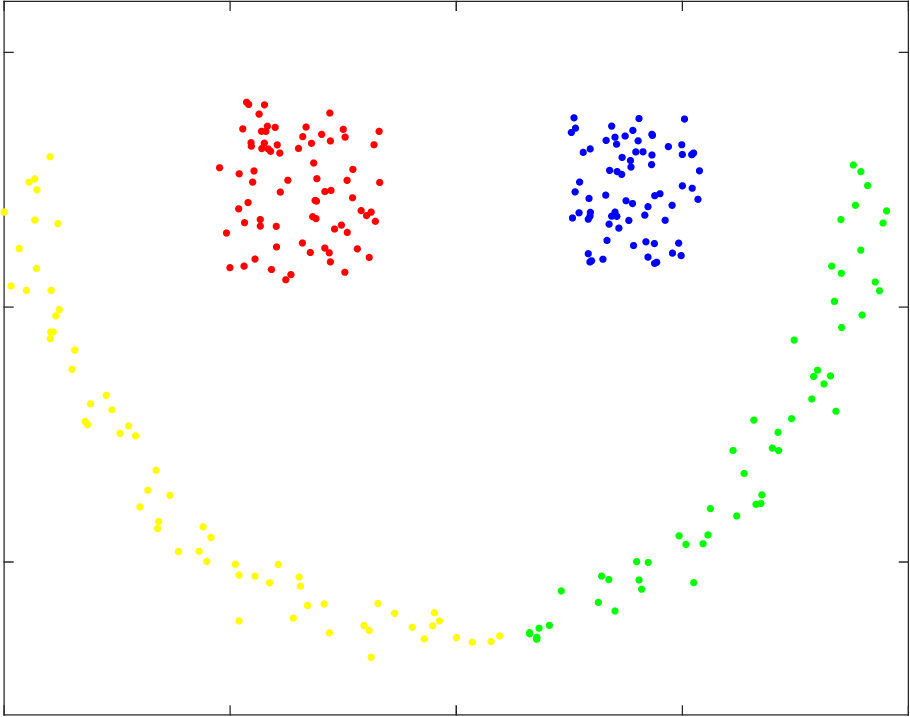}}\quad\quad\quad\quad
\subfigure[\tiny Dataset 2, COSSC, $d$=3,$\ldots$,12]{\includegraphics[height=30mm]{fig/graph2_sc_right}}\\
\vspace{-3mm}
\subfigure[\tiny Dataset 3, SCA, $k=k^*$=3]{\includegraphics[height=30mm]{fig/graph4_sc_right}}\quad\quad\quad\quad
\subfigure[\tiny Dataset 3, SCA, $k$=4]{\includegraphics[height=30mm]{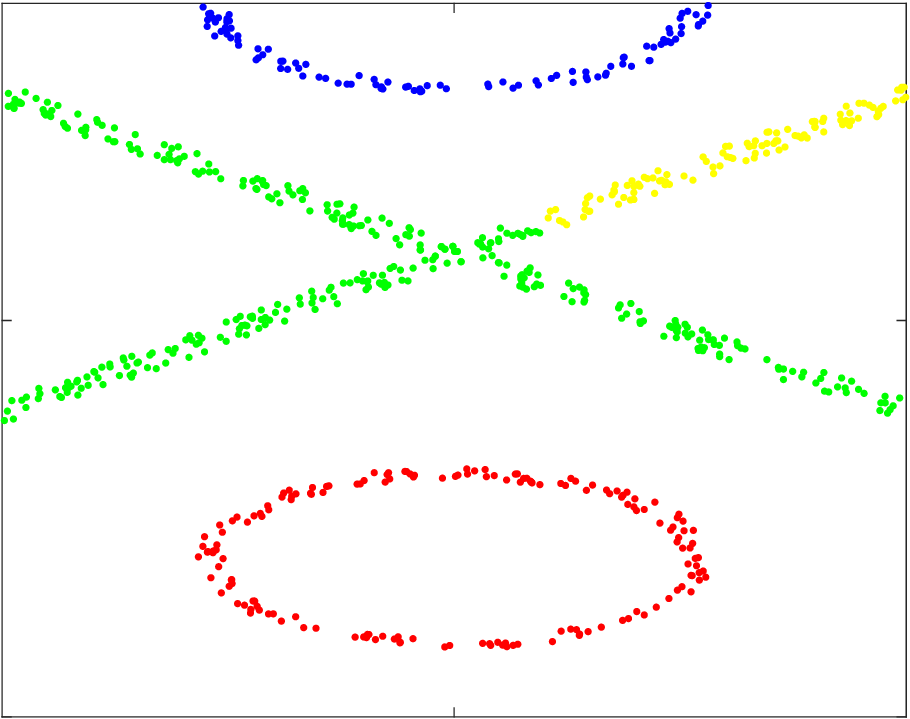}}\quad\quad\quad\quad
\subfigure[\tiny Dataset 3, COSSC, $d$=3,$\ldots$,12]{\includegraphics[height=30mm]{fig/graph4_sc_right}}\\
\vspace{-3mm}
\subfigure[\tiny Dataset 4, SCA, $k=k^*$=3]{\includegraphics[height=30mm]{fig/graph5_sc_right}}\quad\quad\quad\quad
\subfigure[\tiny Dataset 4, SCA, $k$=4]{\includegraphics[height=30mm]{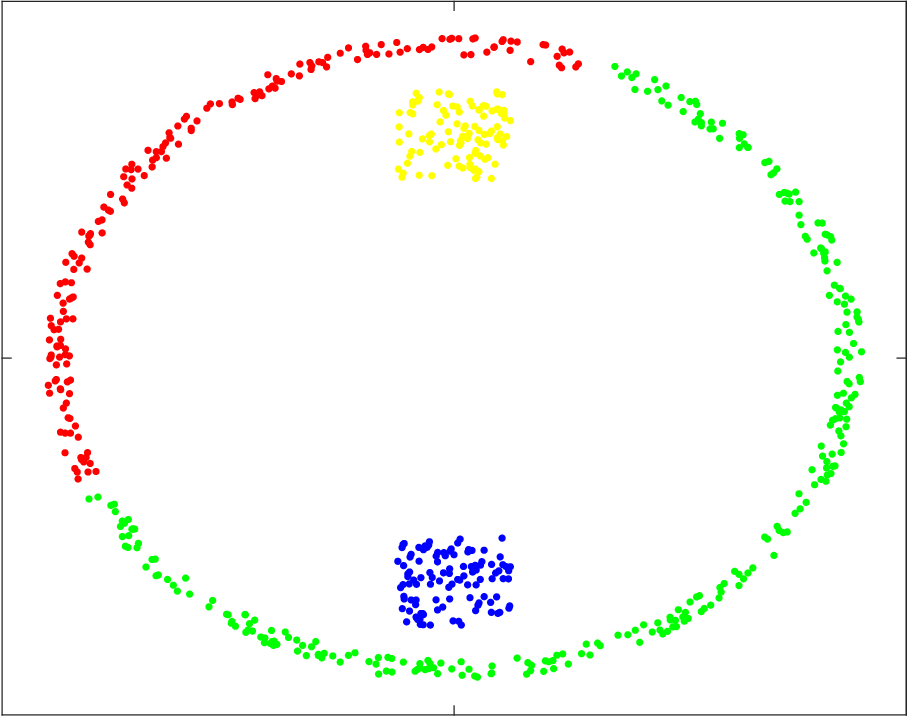}}\quad\quad\quad\quad
\subfigure[\tiny Dataset 4, COSSC, $d$=3,$\ldots$,12]{\includegraphics[height=30mm]{fig/graph5_sc_right}}\\
\vspace{-3mm}
\subfigure[\tiny Dataset 5, SCA, $k=k^*$=2]{\includegraphics[height=30mm]{fig/graph6_sc_right}}\quad\quad\quad\quad
\subfigure[\tiny Dataset 5, SCA, $k$=3]{\includegraphics[height=30mm]{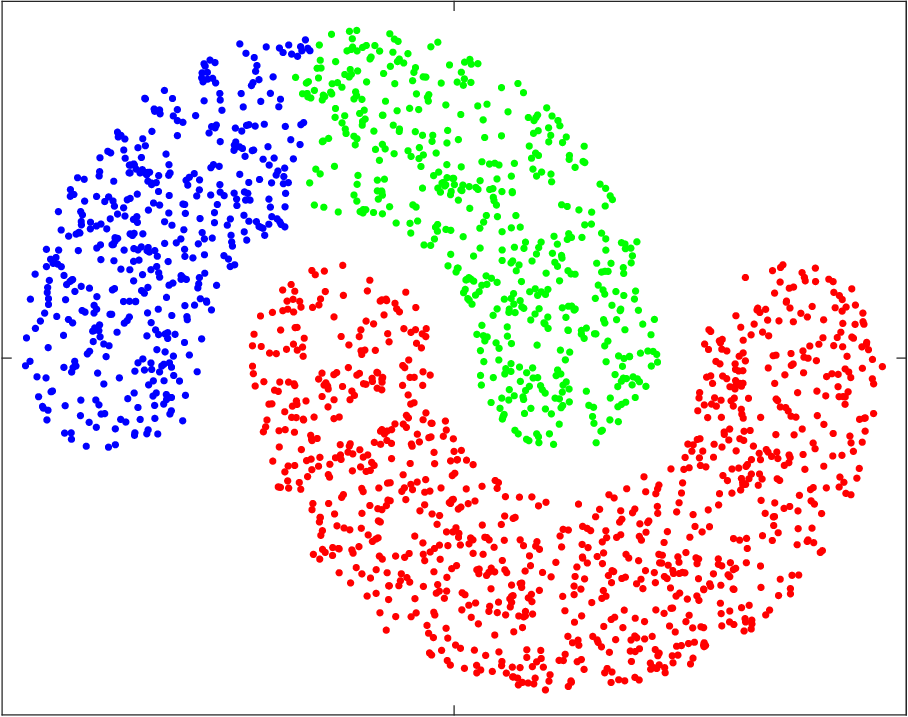}}\quad\quad\quad\quad
\subfigure[\tiny Dataset 5, COSSC, $d$=2,$\ldots$,12]{\includegraphics[height=30mm]{fig/graph6_sc_right}}\\
\vspace{-3mm}
\subfigure[\tiny Dataset 6, SCA, $k=k^*$=5]{\includegraphics[height=30mm]{fig/graph8_sc_right}}\quad\quad\quad\quad
\subfigure[\tiny Dataset 6, SCA, $k$=6]{\includegraphics[height=30mm]{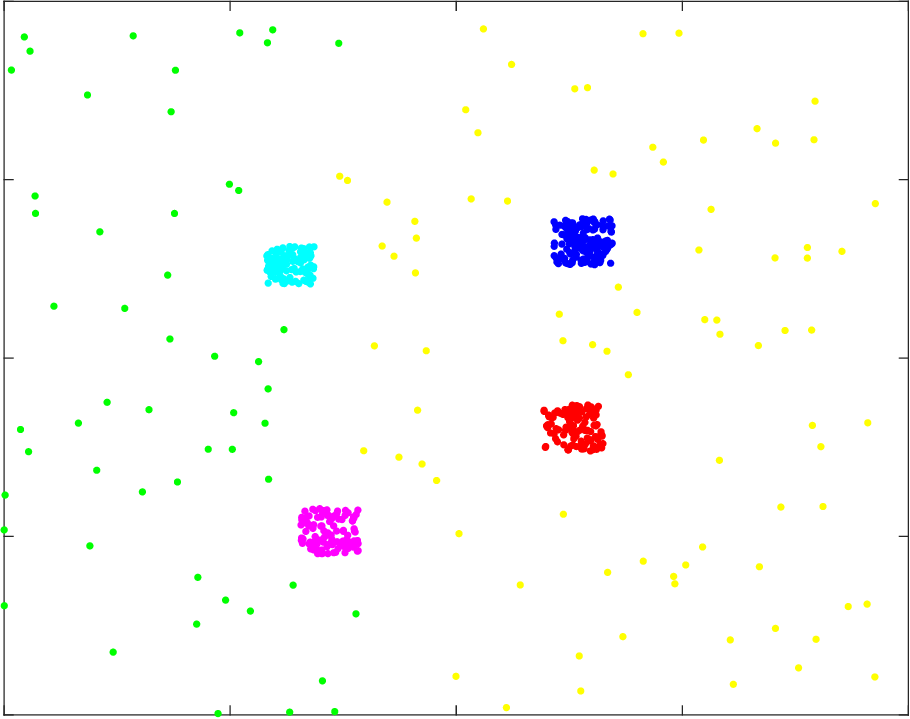}}\quad\quad\quad\quad
\subfigure[\tiny Dataset 6, COSSC, $d$=5,$\ldots$,10]{\includegraphics[height=30mm]{fig/graph8_sc_right}}\\
\vspace{-3mm}
\caption{Comparisons between COSSC and SCA with different input numbers of clusters.}\label{pic:9acc}
\end{figure}

\subsection{Numerical comparisons on document clustering}
\label{sec:docu}

{Now, we }{compare} COSSC with some other methods in solving semi-supervised clustering problems using the TDT2 test set~\cite{cai2005document}, which consists of 10212 news articles from various sources (e.g., NYT, CNN, and VOA) in 1998. There are 450 datasets in the test set. 
For each $k^*\in\{2,3,\ldots,10\}$, there are~50 datasets, each expected to have $k^*$ clusters specified by the test set. These clusters serve as the ground truth or ideal clusters in our tests. We apply each clustering algorithm  to these datasets and evaluate their performance by calculating the average values of  ACC and NMI. For COSSC, we set $\beta={(d-1)}/{n}$ and $p=10$  as suggested in Section \ref{sec:default}. 

In a document clustering task, we generate the must-link constraints in the following way. For each dataset, let $\{P_1,P_2,\ldots,P_{k^*}\}$ be the ideal clusters. 
We randomly select $s\%$ of edges from $\{(i,j)\mid x_i,x_j\in P_l \text{ for some }l\in \{1,2,\ldots, k^*\}\}$, and add must-link constraints on them.
We {refer} to 
$s\%$ as the percentage of  must-link constraints.

\subsubsection{Numerical performances with different parameter $d$}\label{sec:docudiffn}
In this subsection, we compare COSSC with other methods. We set \( s = 5 \) and consider the ideal number of clusters \( k^* \in \{2, 3, 4\} \). For each value of \( k^* \), we test COSSC  with \( d \in \{k^*, \dots, 10\} \), and
for
other algorithms,  we set  \( k=d \) in this comparison; then, each method is applied to the 50 corresponding datasets from TDT2. The average values of ACC and NMI across the 50 datasets are shown in Figure \ref{fig:9beta}, where the horizontal axis represents the input number of clusters. From the results, we observe that (i) as $d$ grows larger than the ideal number of $k^*$,  the numerical performances of COSSC deteriorate at a slower rate compared with other algorithms; (ii)  {COSSC achieves the best performance among all methods in terms of ACC and NMI in most cases.}

\begin{figure}[tbp!]
\centering
\subfigure{\includegraphics[height=35mm]{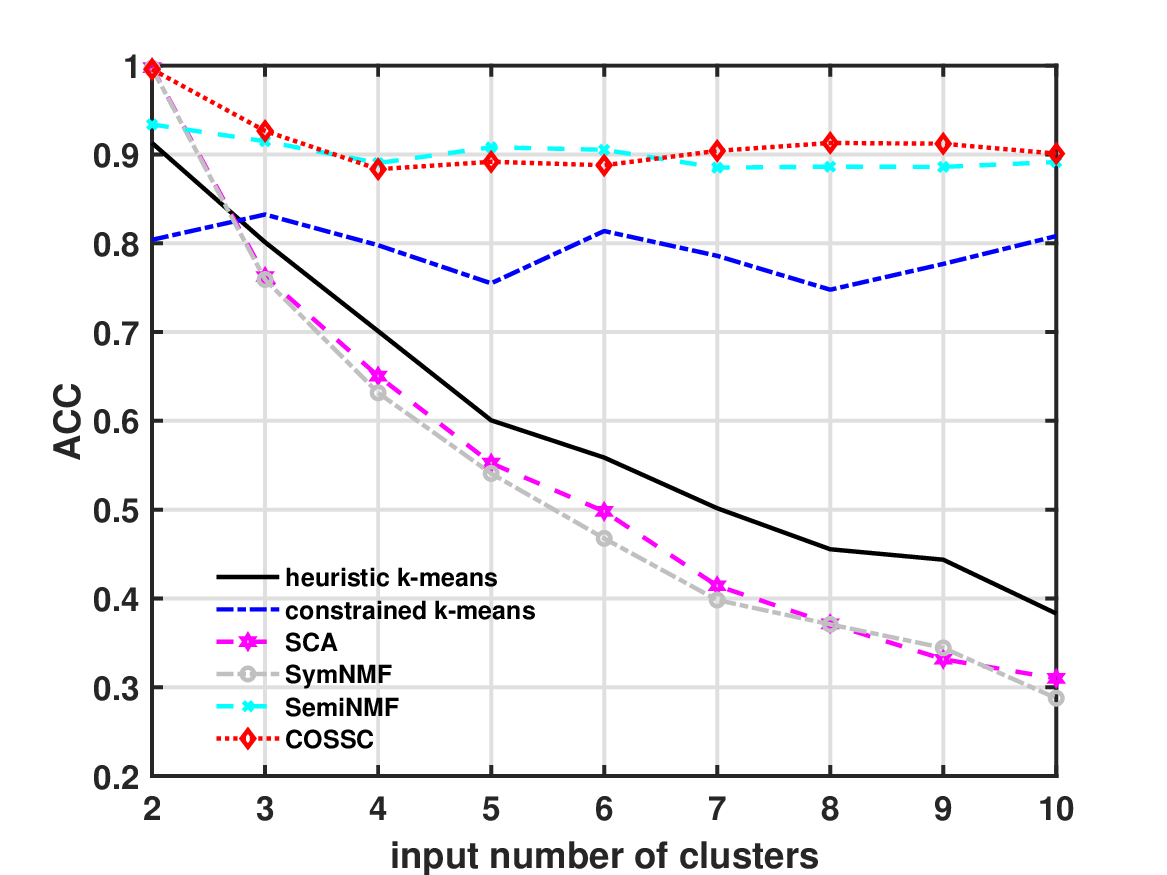}}\quad\quad
\subfigure{\includegraphics[height=35mm]{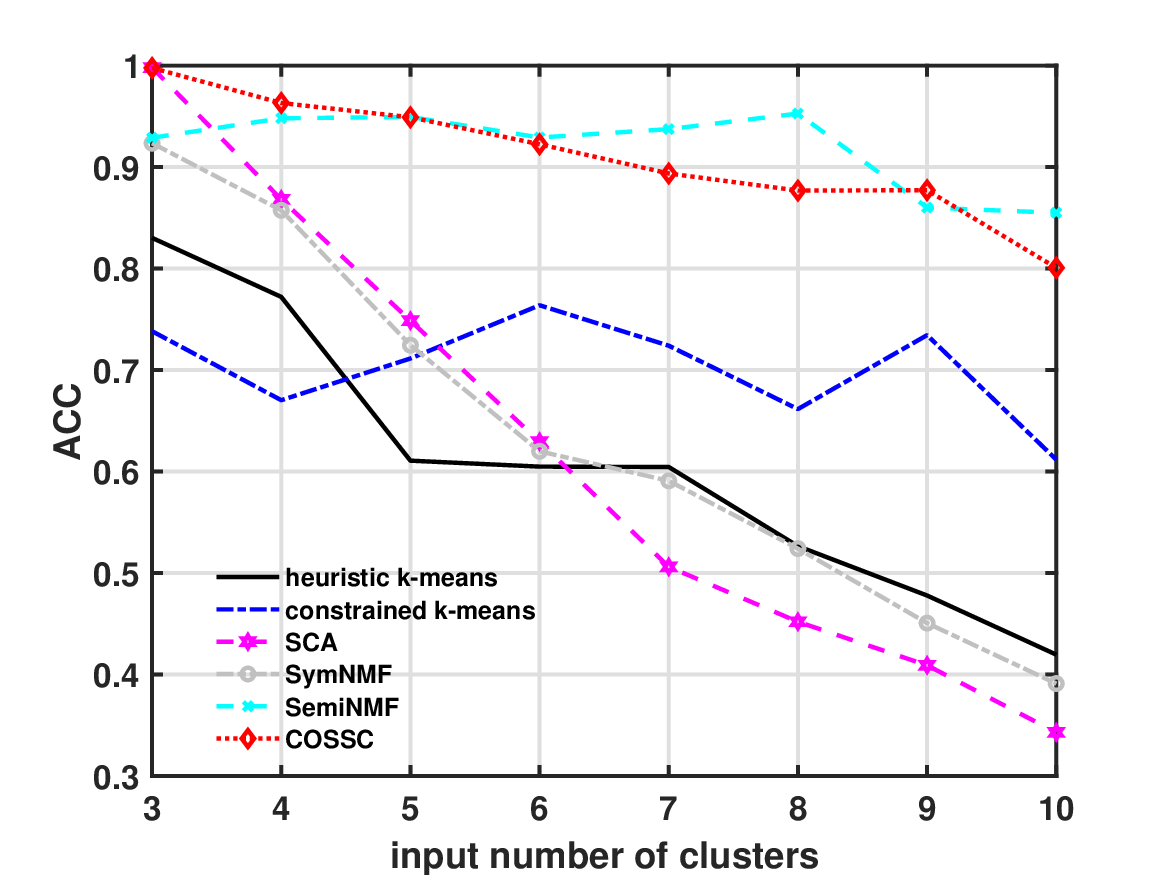}}\quad\quad
\subfigure{\includegraphics[height=35mm]{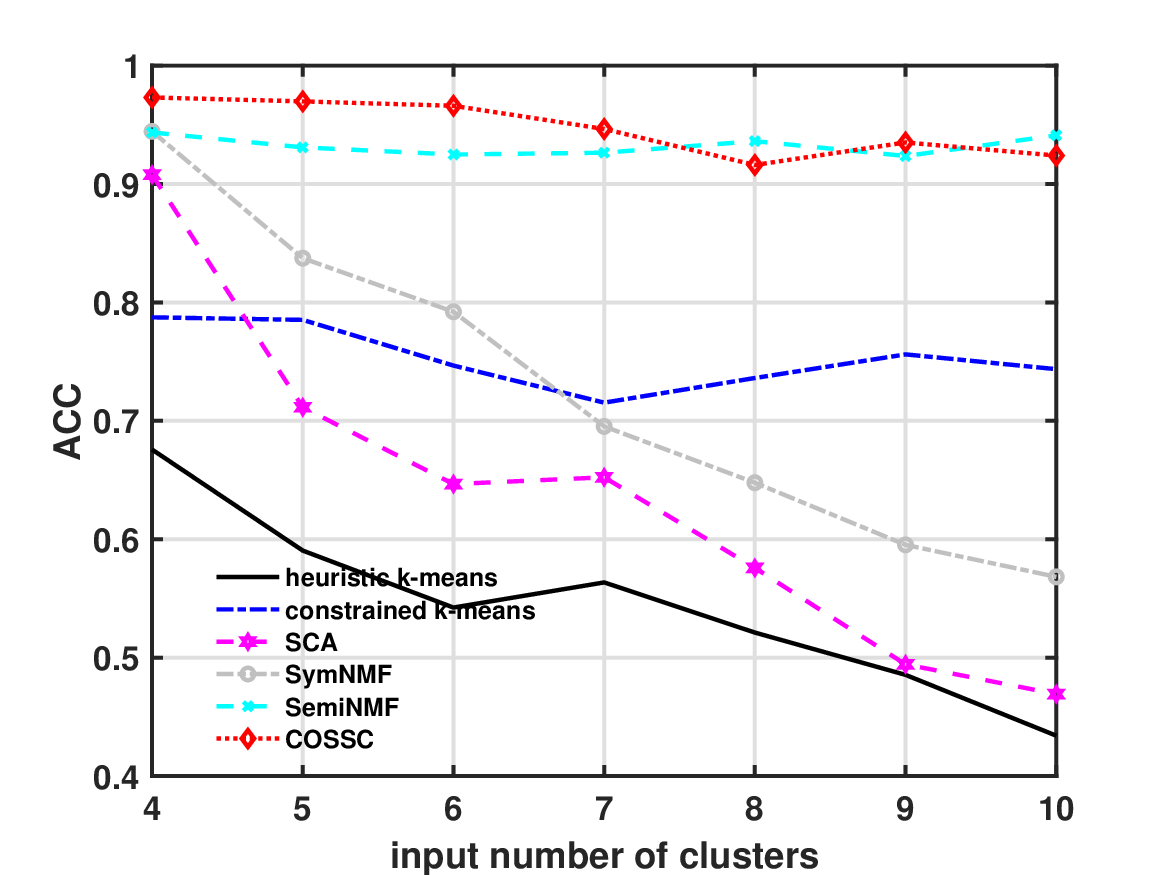}}\\
\setcounter{subfigure}{0}
\vspace{-3mm}
\subfigure[$k^*=$ 2]{\includegraphics[height=35mm]{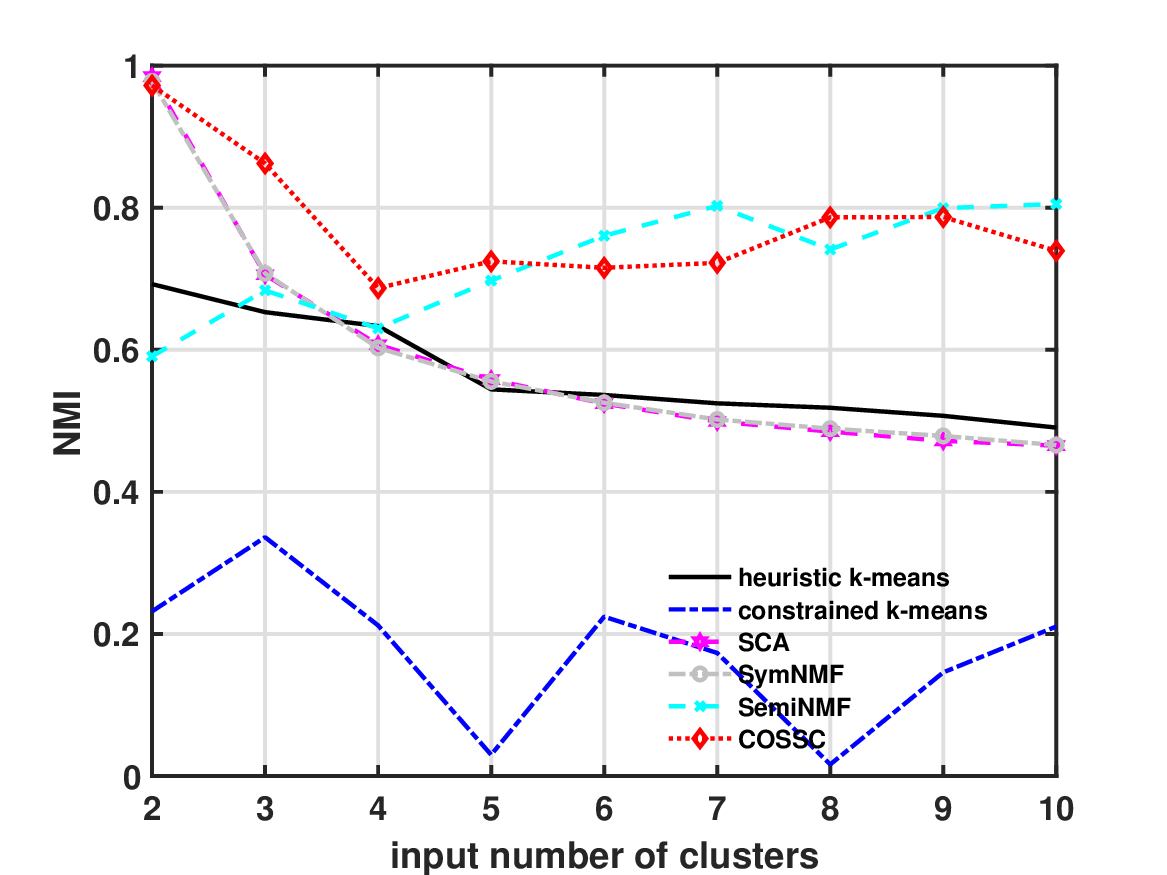}}\quad\quad
\subfigure[$k^*=$ 3]{\includegraphics[height=35mm]{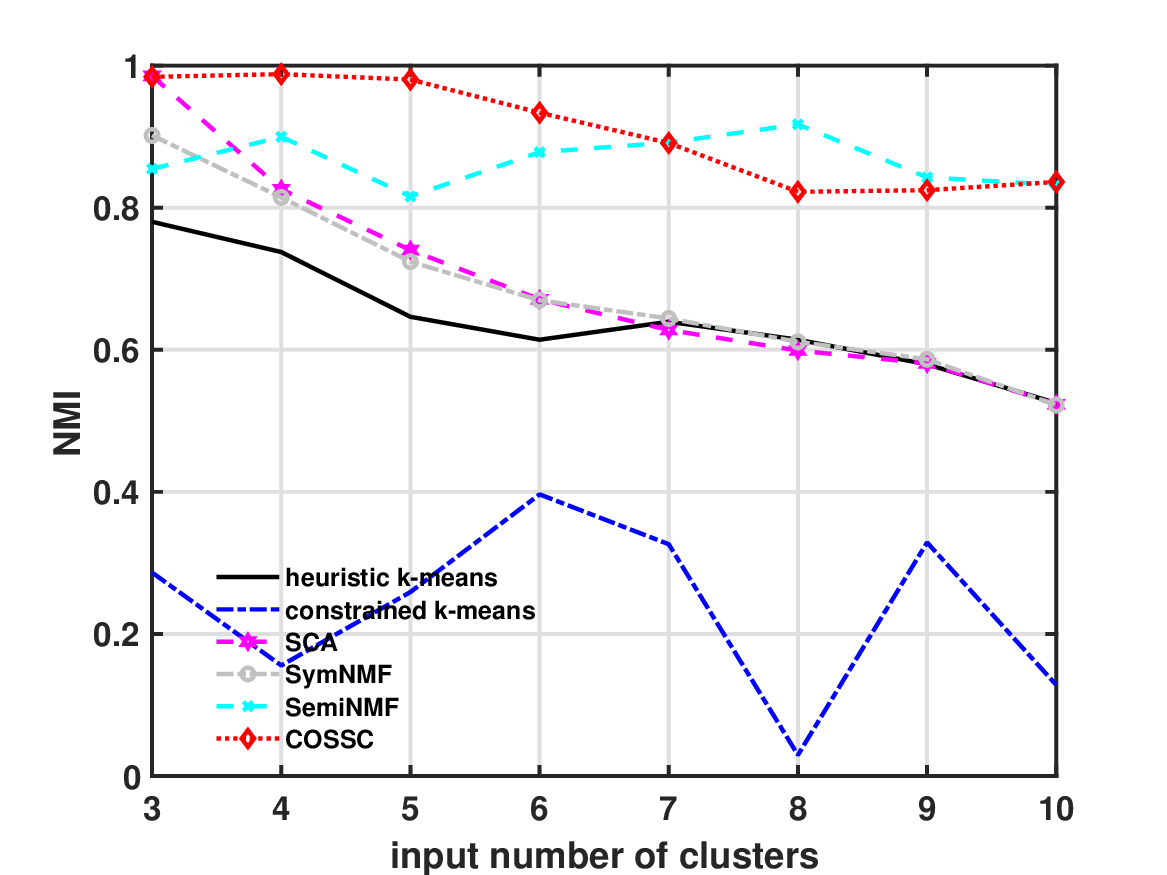}}\quad\quad
\subfigure[$k^*=$ 4]{\includegraphics[height=35mm]{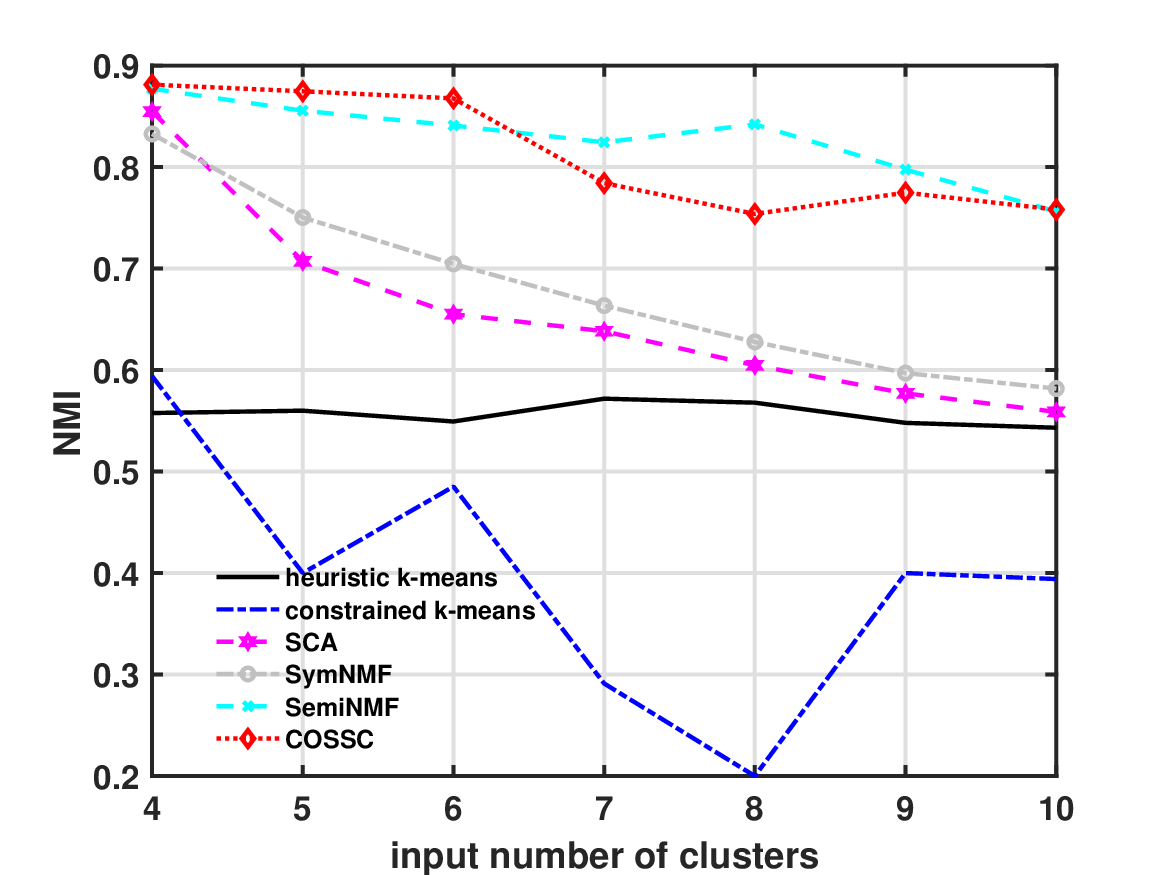}}\quad\\
\caption{Comparisons of ACC and NMI among COSSC and other algorithms with $d=k\in\{k^*,\ldots,10\}$.}\label{fig:9beta}
\end{figure}

\subsubsection{Numerical performances with different amounts of must-link constraints}
\label{sec:442}

Now, we compare COSSC with the  five methods mentioned in Section \ref{sec:default} {across} varying percentages of must-link constraints. Specifically, we set $k^*\in\{2,3,4\}$ and $s\in\{5,10,\ldots,50\}$. For each~$k^*$, we apply each clustering algorithm with $d=k=k^*$ to the 50 corresponding datasets on TDT2, and present the average results of ACC, NMI and RMV in Figure~\ref{fig:9beta2}.   
The first two rows of Figure~\ref{fig:9beta2} show that COSSC outperforms the other methods in terms of both ACC and NMI.  
In the last row, we observe  that COSSC achieves an RMV of zero for all $k^*$ and~$s$. 
The observation indicates that the graph associated with the output of COSSC satisfies the must-link constraints, even though our settings of $\beta$ and $p$ does not necessarily satisfy the assumption in Theorem \ref{thm:mustlink2}.
\begin{figure}[htbp!]
\centering
\subfigure{\includegraphics[height=35mm]{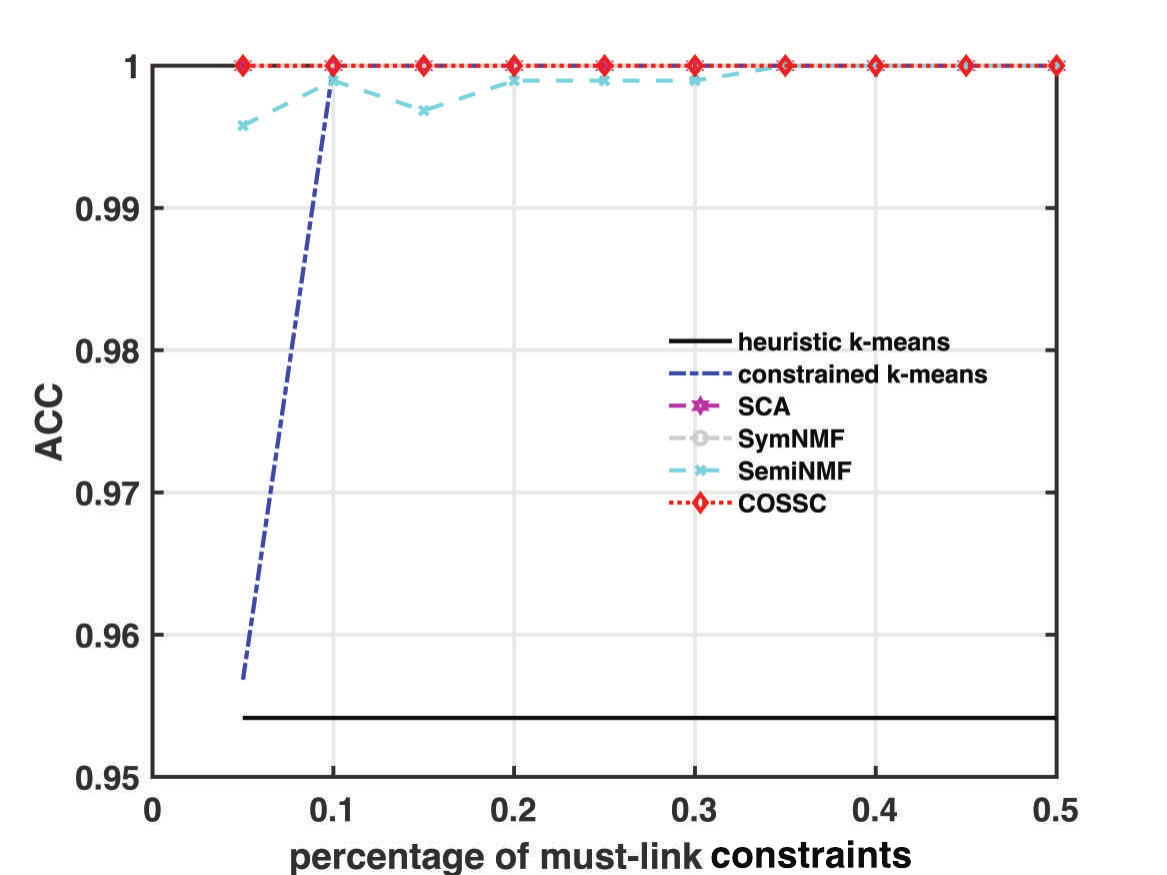}}\quad\quad
\subfigure{\includegraphics[height=35mm]{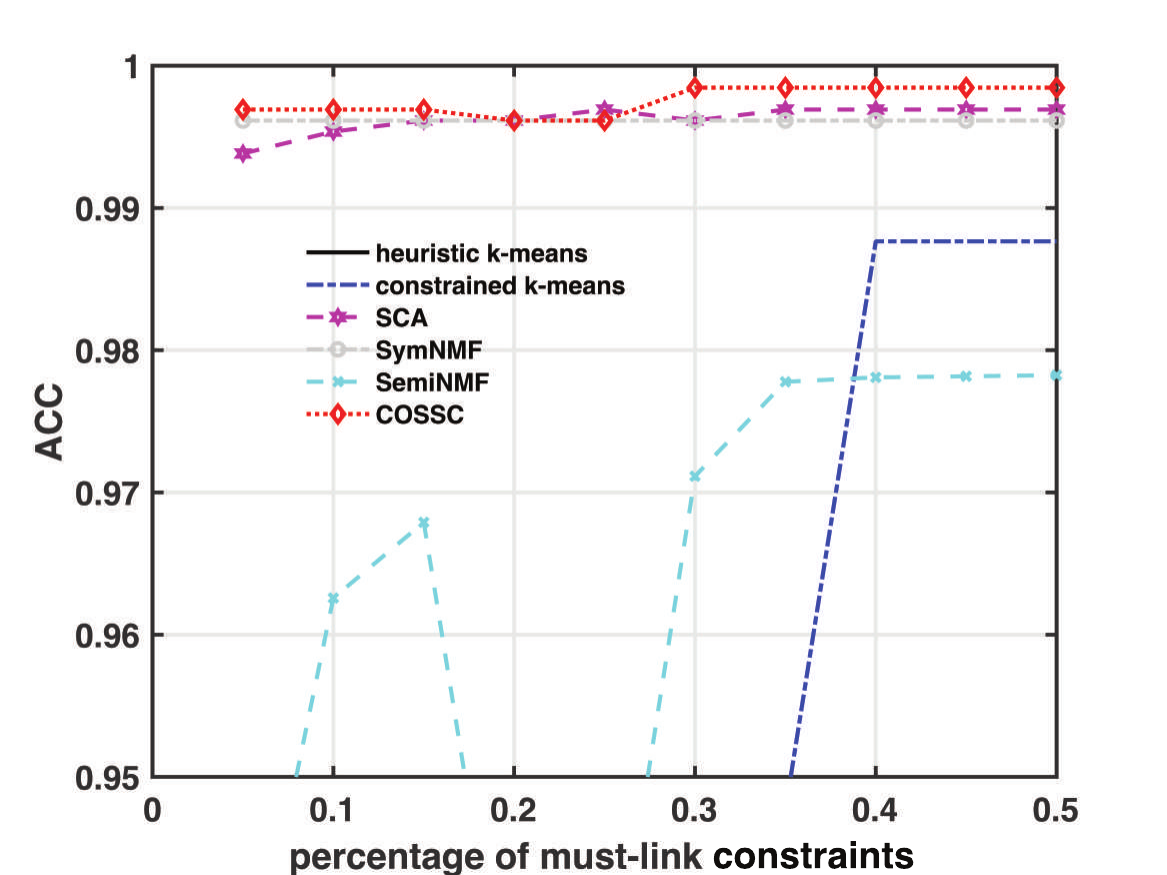}}\quad\quad
\subfigure{\includegraphics[height=35mm]{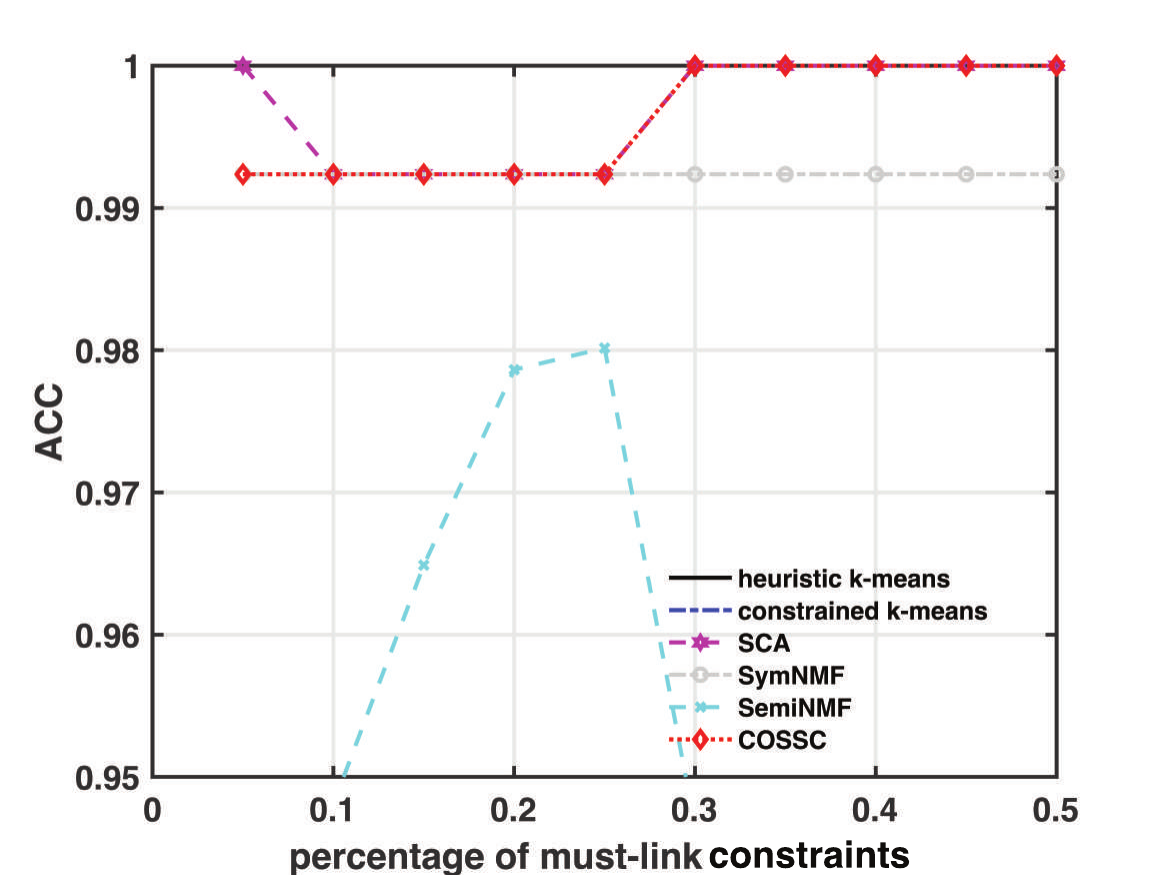}}\quad\\
\vspace{-3mm}
\subfigure{\includegraphics[height=35mm]{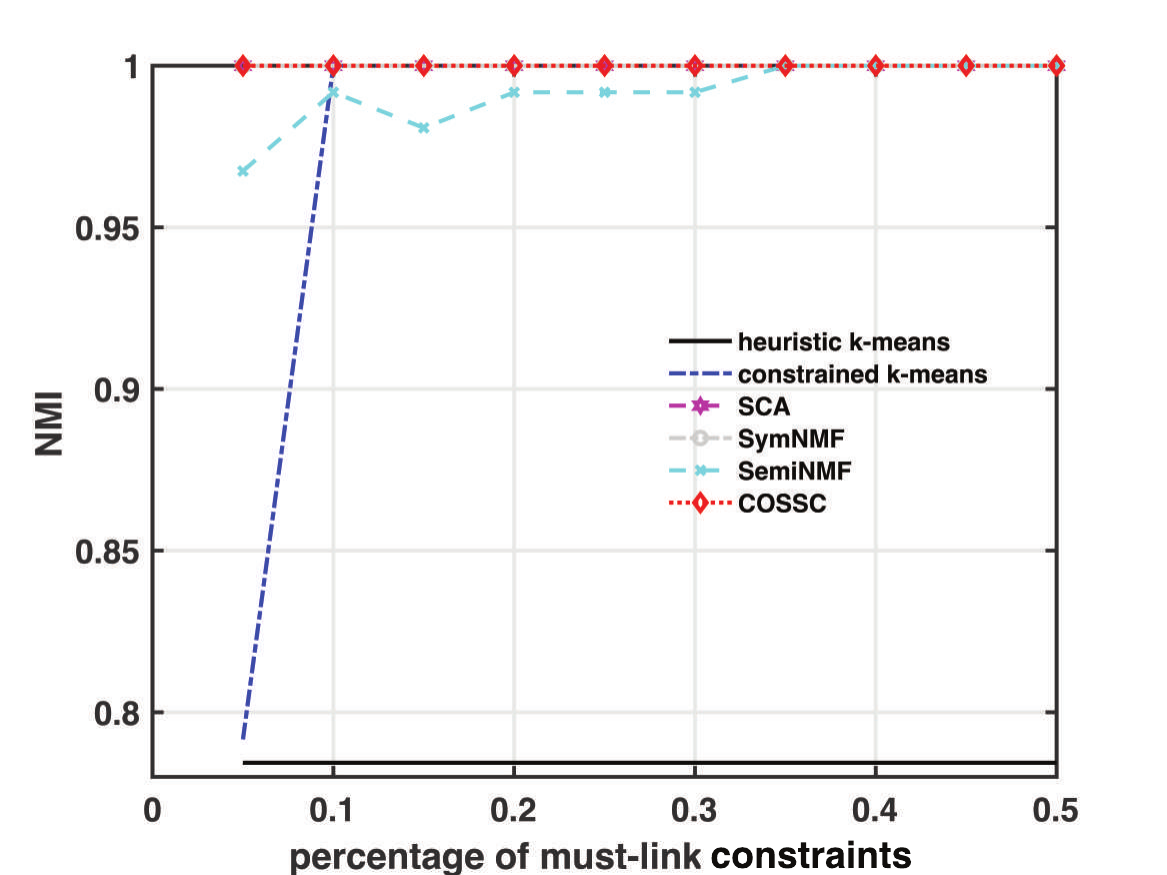}}\quad\quad
\subfigure{\includegraphics[height=35mm]{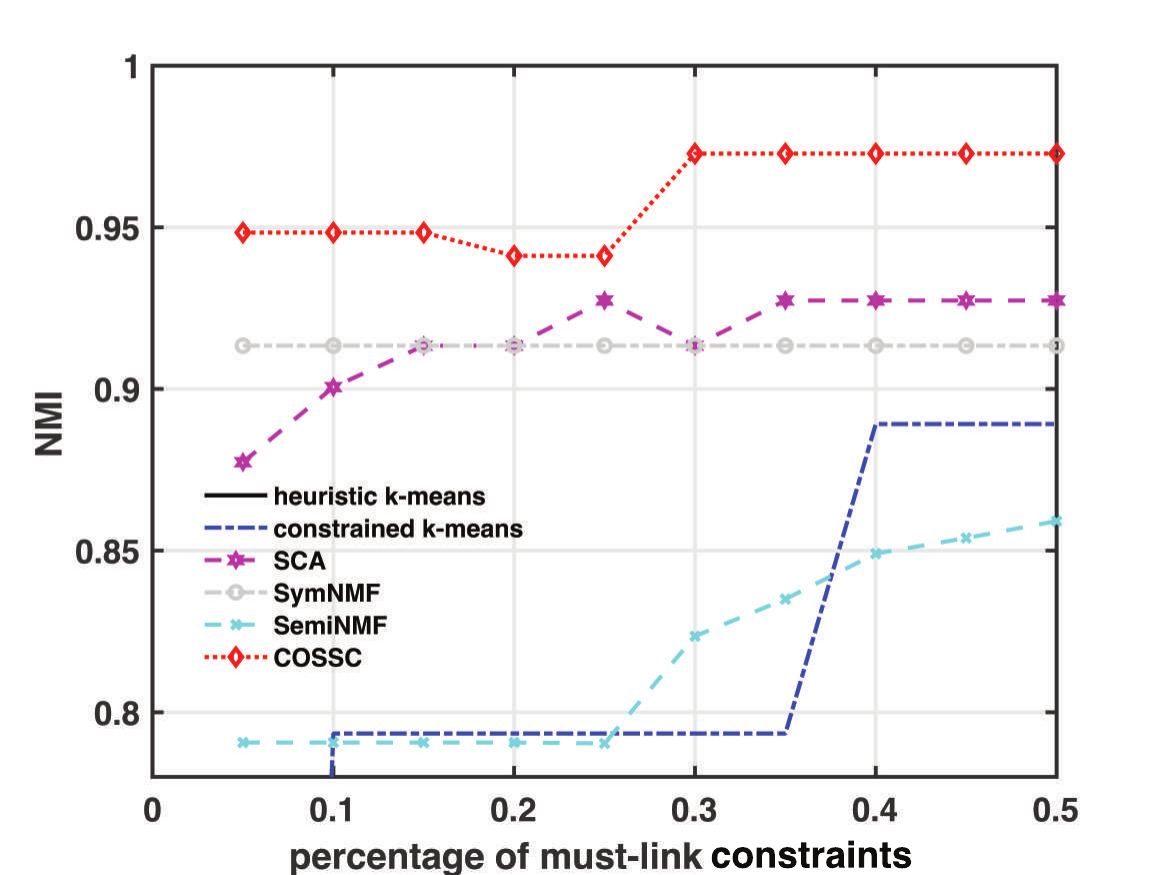}}\quad\quad
\subfigure{\includegraphics[height=35mm]{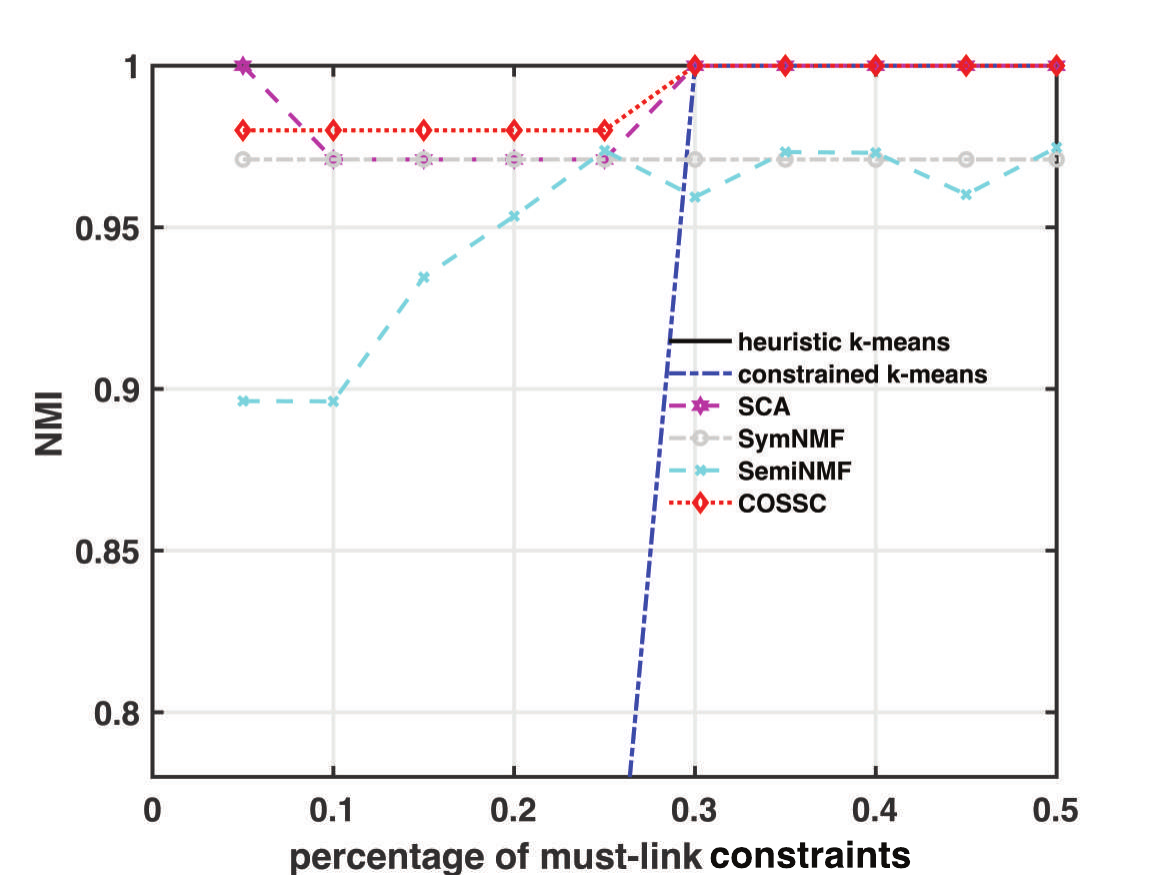}}\quad\\
\setcounter{subfigure}{0}
\vspace{-3mm}
\subfigure[$k^*$= 2]{\includegraphics[height=35mm]{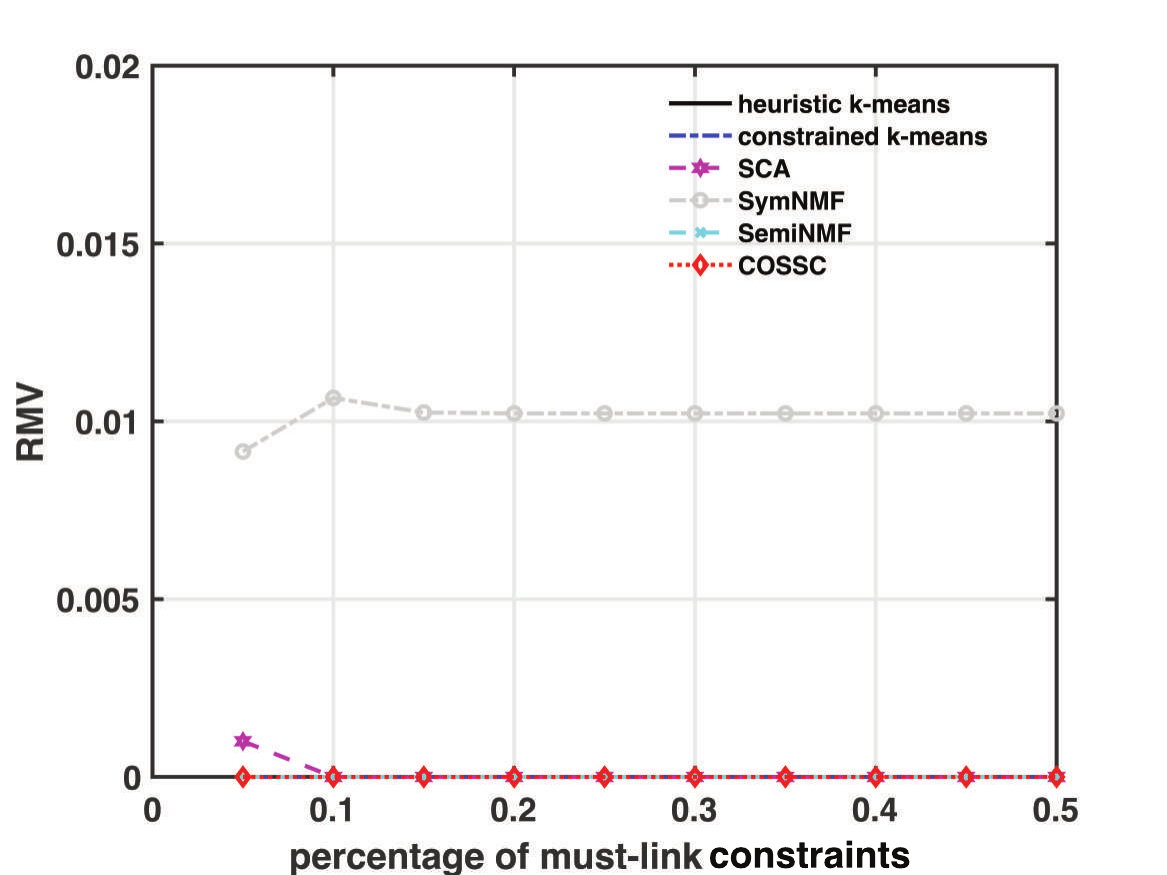}}\quad\quad
\subfigure[$k^*$= 3]{\includegraphics[height=35mm]{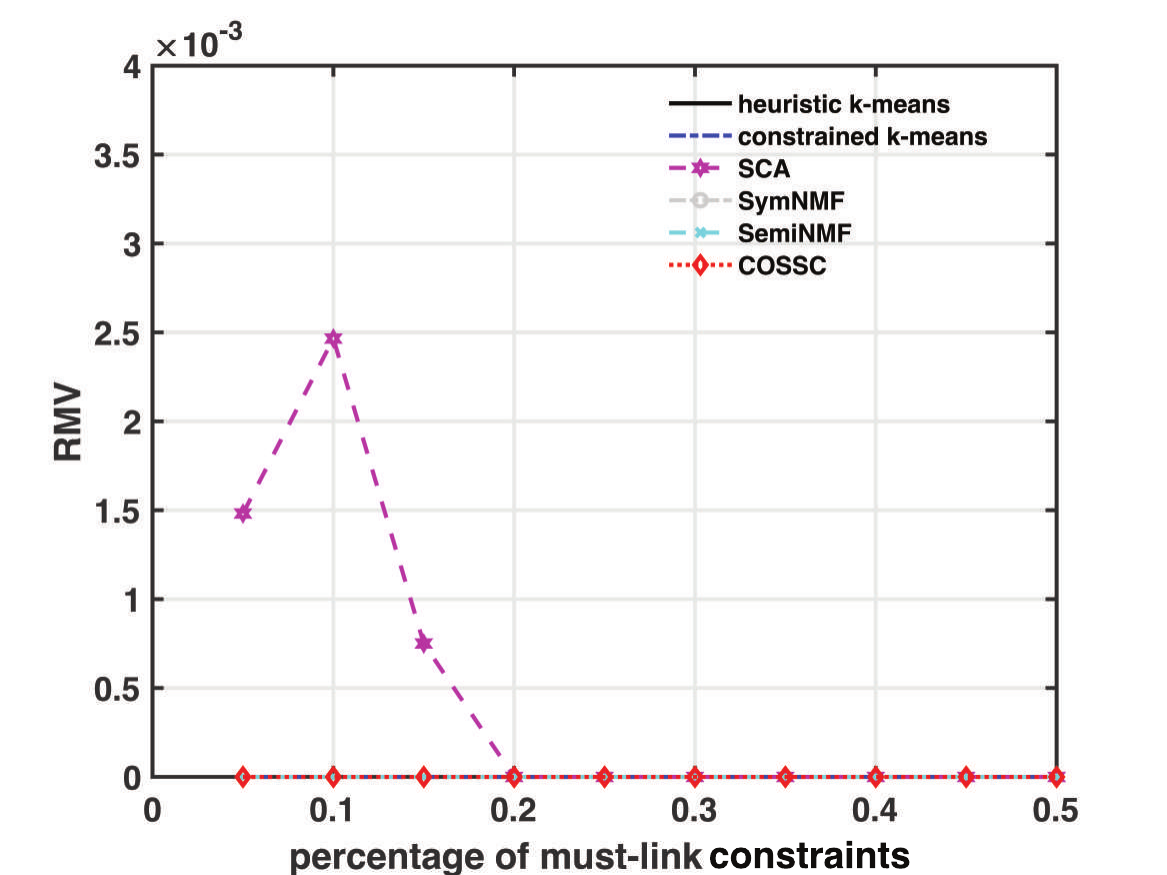}}\quad\quad
\subfigure[$k^*$= 4]{\includegraphics[height=35mm]{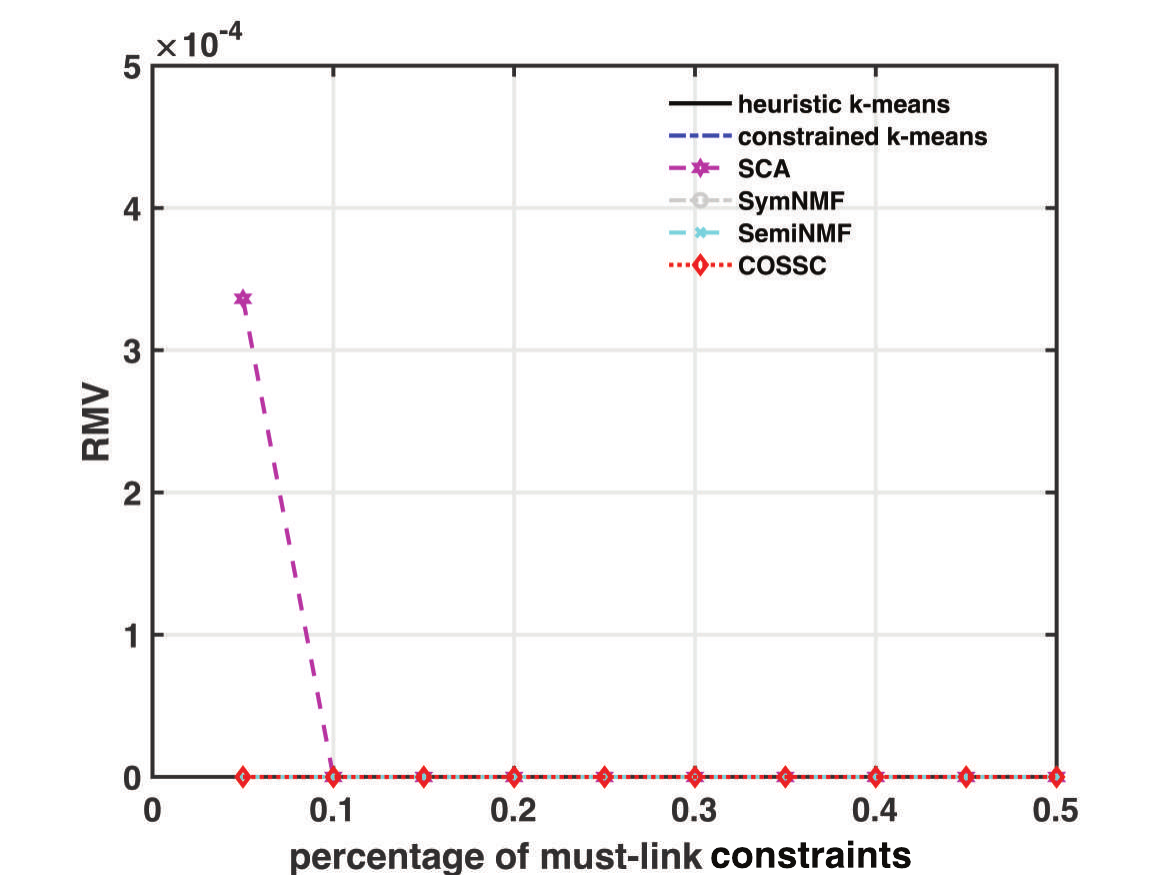}}
\caption{Comparisons of ACC, NMI, and RMV with different percentages of must-link constraints and $d=k=k^*$.}\label{fig:9beta2}
\end{figure}



\subsubsection{Comparisons of CPU time}\label{sec:cpu}

{In this subsection, we aim to show that COSSC is not only effective but also evidently more efficient compared with the other algorithms mentioned above.} We set the percentage $s$ of must-link constraints to be $25$. For each $k^*\in\{2,3,\ldots, 10\}$, we apply each clustering algorithm with $d=k=k^*$ to the 50 corresponding dataset.    Figure~\ref{fig:time}
shows how the average CPU time changes as $k^*$ increases. We observe that COSSC is the best among all the methods.
In addition to the CPU time, COSSC also performs well in terms of the values of ACC, NMI, and RMV, which are presented in Appendix~\ref{sec:overall}. 

\begin{figure}[tbp!]
\centering
\subfigure{\includegraphics[height=55mm]{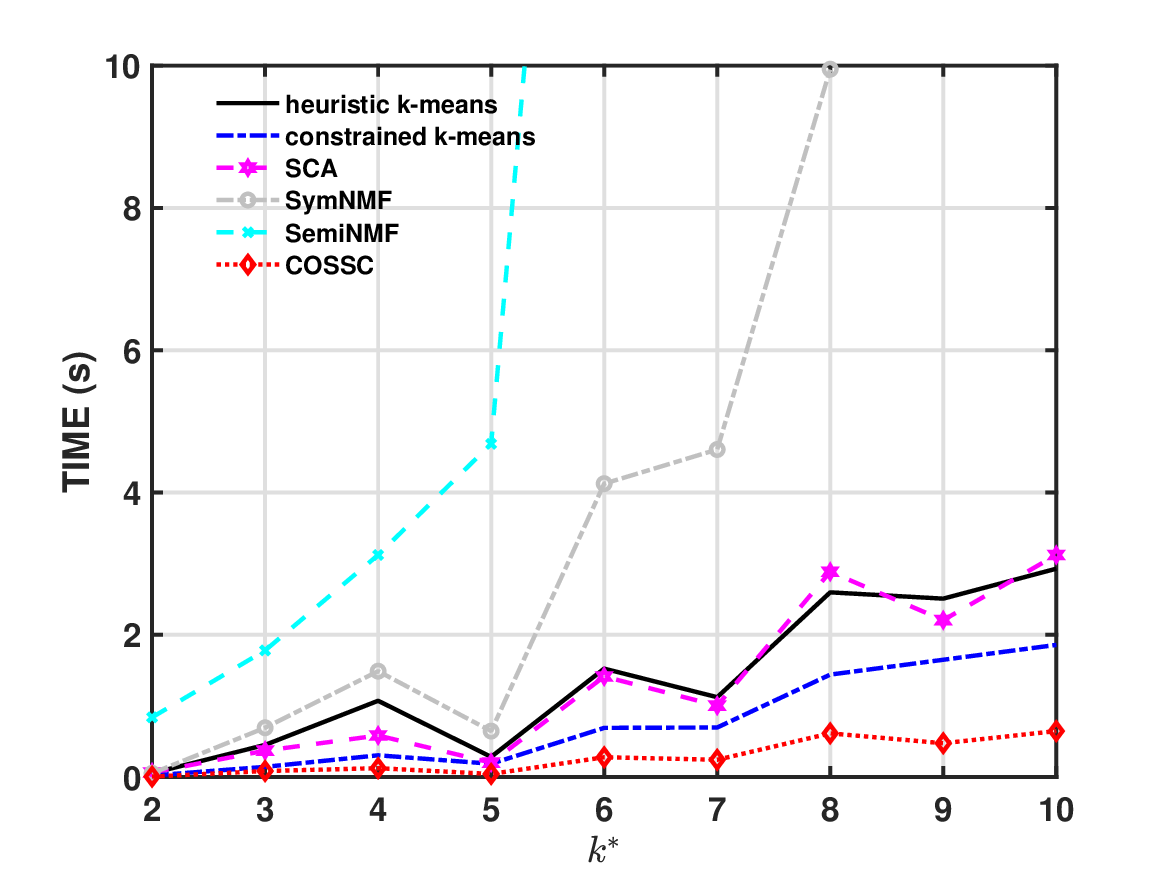}}
\caption{Comparison of TIME among COSSC and other algorithms with $d=k=k^*$.}\label{fig:time}
\end{figure}

\section{Conclusions and discussions}


\revise{
	In this paper, we mainly propose a semi-supervised clustering approach that eliminates the need for the ideal number of clusters as an input and ensures the satisfaction of must-link constraints under mild requirements on model parameters. By framing semi-supervised spectral clustering as a graph partitioning problem, we first introduce a novel continuous optimization model, analyze its theoretical properties, and then develop a block coordinate descent algorithm with guaranteed convergence to solve it. The resulting clusters can be directly obtained by applying a simple search algorithm to the solution of our model. Extensive numerical experiments demonstrate that our method outperforms several existing semi-supervised clustering techniques on both synthetic and real-world document clustering datasets.}
	
\revise{
	Finally, we note that this paper does not address another type of supervisory information: the cannot-link constraint (see \cite{zhu2005semi1}), 
	which prohibits two data points from belonging to the same cluster. Satisfying cannot-link constraints requires ensuring that no path exists between $x_i$ and $x_j$ for all cannot-link pairs $(i,j)$, a task that is inherently challenging due to its combinatorial nature. 
	 As demonstrated in \cite{davidson2007complexity}, even determining the feasibility of cannot-link constraints is an NP-hard problem. 
	 To tackle this challenge, a potential heuristic approach could involve extending~\eqref{eq:Aconstruct} by assigning negative weights to cannot-link constraints and replacing the Laplacian matrix with the signed Laplacian matrix~\cite{harary1953notion}. 
	 We leave this extension for future work.
}

\section*{Acknowledgements} 
The work of Xin Liu was supported in part by the National Key R\&D Program of China 2023YFA1009300, 
the National Natural Science Foundation of China (12125108, 12021001, 12288201), and RGC grant JLFS/P-501/24 for the CAS AMSS-PolyU Joint Laboratory in Applied Mathematics. The work of Zaikun Zhang was supported in part by the Hong Kong RGC grants 15305420, 15306621, and 15308623. The work of Michael Ng was supported in part by the Hong Kong RGC grants 17201020, 17300021, and C7004-21GF, and NSFC-RGC joint grant N-HKU76921.

\appendix

\section{The derivative of $f$ with respect to $Z$}
\label{appen:derive}
We use $\nabla_{Z} f(Z, H)$ to denote the derivative of $f$ with respect to $Z$.
Recalling the definition of $f$, we have 
\begin{equation*}
\begin{aligned}
	f(Z, H) &=  \langle \LL(A \circ Z), HH^\TT\rangle - \beta \langle
	Z, \bar{A}\rangle =  \langle  A\circ Z,  \LL^*(HH^\TT) \rangle- \beta \langle
	Z, \bar{A}\rangle
	\\ &=  \langle   Z, A\circ \LL^*(HH^\TT) \rangle - \beta \langle
	Z, \bar{A}\rangle = \langle  Z,  A\circ\LL^*(HH^\TT) - \beta \bar{A}\rangle,
\end{aligned}
\end{equation*}
where $\LL^*$ represents the Hermitian adjoint operator of $\LL$.  
In addition, for any matrices $M, N\in\R^{n\times n}$, it holds that 
\begin{align*}
\langle M, \LL^* (N) \rangle = &\langle \LL(M), N \rangle 
=  \sum_{i=1}^n\sum_{j\neq i} \LL(M)_{ij} N_{ij} + \sum_{i=1}^n \LL(M)_{ii} N_{ii}\\
=& -\sum_{i=1}^n\sum_{j\neq i} M_{ij}N_{ij}
+ \sum_{i=1}^n\Bigg(\sum_{j\neq i} M_{ij}\Bigg) N_{ii}
=  \sum_{i=1}^n\sum_{j=1}^n M_{ij}(N_{ii}-N_{ij}),
\end{align*}
which implies $(\LL^* (N))_{ij}=N_{ii}-N_{ij}$ for all $i,j=1,2,\ldots,n$.
Combining these facts, we have
\begin{align}
\label{eq:derivef}
(\nabla_{Z} f(Z, H))_{ij} = A_{ij}(\LL^*(HH^\TT))_{ij} - \beta \bar{A}_{ij} = A_{ij}((HH^\TT)_{ii}- (HH^\TT)_{ij}) - \beta \bar{A}_{ij}.
\end{align}

\section{Proof of Theorem \ref{thm:thmor}}\label{appen:1}

 {The following lemmas will be used in the proofs of Theorem \ref{thm:thmor}. 
	\begin{lemma}[{\cite[Proposition 2]{von2007tutorial}}]
		\label{lem:b1}
		Given two matrices $A_1,A_2\in\R^{n\times n}$.
		If $A_1\ge A_2$, then $\LL(A_1)\succeq \LL(A_2)$. In particular, if $A_1\ge 0$, then $\LL(A_1)\succeq 0$.
\end{lemma}}

\begin{lemma}
		\label{lem:b2}
	Given two matrices $A_1,A_2\in\R_+^{n\times n}$.
	If it holds $\supp(A_1) =  \supp(A_2) $, then $\rank(\LL(A_1 )) = \rank(\LL(A_2)).$
\end{lemma}

The proof of Lemma \ref{lem:b2} is similar to \eqref{eq:rank22}. We now present the proof of Theorem \ref{thm:thmor}.
\begin{proof}
Since $ \LL(A\circ Z^*)\succeq 0 $, it holds that 
$
\tr ((H^*)\zz \LL(A\circ Z^*) H^*) \geq 0$.
We then consider two cases based on the value of this trace.

If $
\tr ((H^*)\zz \LL(A\circ Z^*) H^*) >0$,
the optimality of $H^*$ together with the Rayleigh-Ritz theorem imply that $\rank(\LL({A} \circ Z^*) ) > n-d$.  In this case, (a) holds, and (b) does not hold.

Now we consider the case when $
\tr ((H^*)\zz \LL(A\circ Z^*) H^*) =0$. In this case, we will show that~(a) does not hold, and (b)  holds.
To this end, we first prove that there exists a matrix $\bar{Z}$ satisfying 
$
\bar{Z} \ge Z^*$, $\bar{Z} \in\snA \cap\{0, 1\}^{n\times n}$, and $ \rank(\LL({A} \circ \bar{Z}) ) = n-d. 
$ 
We will achieve this by constructing a finite sequence of matrices $\{Y_l\}\subset \snA \cap \{0, 1\}^{n\times n}$ with $Y_0= \lceil Z^* \rceil$, such that 
\begin{equation}
	\rank(\LL(A\circ Y_{l})) \le \rank(\LL(A \circ Y_{l+1}) ) \le \rank(\LL(A\circ Y_{l})) +1
	\label{eq:r2}	
\end{equation}
for all $l\geq 0$ before termination. For each $l\ge 0$, if $Y_l \neq \mathrm{sign}(A)$, we choose $(\bar{i},\bar{j})\in \supp(A)\setminus \supp(Y_l)$, set $(Y_{l+1})_{\bar{j}\bar{i}}=(Y_{l+1})_{\bar{i}\bar{j}}=1$, and the other entries of $Y_{l+1}$ to be the same as $Y_l$.
Otherwise, if $Y_l= \mathrm{sign}(A)$, we terminate the construction, which will happen in finitely many iterations since $\supp(A)$ is finite.
For all $Y_{l}$,
it  holds that $Y_{l}\in \snA \cap \{0, 1\}^{n\times n}$ and $ Y_l\geq Z^*$.  
Consider graphs $\mathcal{G}(X, {A\circ Y_l}) $ and $\mathcal{G}(X, {A\circ Y_{l+1}})$ for any $l\geq 0$ before termination. The latter graph has one more edge. Hence, the number of connected subgraphs in the latter is not less than   in the former,
the difference being at most 1.
This leads to \eqref{eq:r2} according to~\cite[Theorem 2.1]{Mohar_1991}.
The existence of $\bar{Z}$ is then justified, because the sequence $\{Y_l\}$ begins with $\lceil Z^* \rceil$ and ends with $\sign(A)$, and $\rank(\LL({A} \circ \lceil Z^* \rceil) )= \rank(\LL({A} \circ Z^*) ) \leq n-d\leq \rank(\LL({A}) ) = \rank(\LL({A} \circ {\sign(A)}) )$.  Here, the equalities result from Lemma \ref{lem:b2}, the first inequality follows from the Rayleigh-Ritz theorem and $
\tr ((H^*)\zz \LL(A\circ Z^*) H^*) =0$, and the second inequality uses the assumption on $d$.
 
Now, let $\bar{H}$ be a solution to
$$
\min_{H\zz  H = I_{d}} \tr (H\zz  \LL({A}\circ \bar{Z}) H).
$$
Since $\rank(\LL({A} \circ \bar{Z}) ) = n-d$, it holds 
$$\tr (\bar{H}\zz\LL({A}\circ \bar{Z}) \bar{H})=0=\tr((H^*)\zz \LL({A}\circ
Z^*) H^*), 
$$
where the first equality follows from the Rayleigh-Ritz theorem, and the second one is  our setting in this case.
Consequently, the global optimality of $(Z^*, H^* ) $ for problem~\eqref{eq:model} implies
$
\tr(\bar{A} \bar{Z}) \le  \tr(\bar{A}Z^*).
$ 
In other words, $\sum_{i,j}(\bar{A}\circ \bar{Z})_{ij} \le \sum_{i,j}(\bar{A}\circ {Z^*})_{ij}$.
Besides, $\bar{A} \circ \bar{Z} \ge \bar{A} \circ Z^*$ by
 $\bar{Z} \ge Z^*$ and $\bar{A}\ge 0$.
Thus
$
\bar{A} \circ \bar{Z} =\bar{A} \circ Z^*.
$
Since $\bar{Z},Z^*\in\snA  = \mathcal{S}^n_{\bar A}$, we conclude that
$
\bar{Z} = Z^*.
$
Hence we have $\rank(\LL({A} \circ Z^*) ) = \rank(\LL({A} \circ \bar{Z}) ) = n-d$ and $Z^* \in \{0, 1\}^{n\times n}$. Thus case (b) holds.

The proof is completed by combining the two cases.
\end{proof}

\section{Proof of Theorem \ref{thm:bigb}}\label{appen:2}

\begin{proof} 
(a) We define $\phi(Z) := \min_{H\zz  H=I_{d}}f(Z, H).$
By the Rayleigh-Ritz theorem, we have
$$
\phi(Z) = \sum_{i=1}^d \lambda_i \left(\LL(A\circ Z)\right) - \beta \tr(\bar{A}Z). 
$$
Noticing 
$\tr(\bar{A}Z)
=\tr (\LL(\bar{A}\circ Z))
$ by the definition of~$\LL$,
we have
\begin{align*}
	\small
	\phi(Z) &= \tr(\LL(A\circ Z) ) - \beta \tr (\LL(\bar{A}\circ Z)) -\sum\limits_{i=d+1}^{n}\lambda_i (\LL(A\circ Z)) \\
	&= \tr(\LL((A - \beta \bar{A}) \circ Z)) -\sum\limits_{i=d+1}^{n}\lambda_i (\LL(A\circ Z)).
\end{align*}

Suppose for contradiction that $Z^*\neq  \sign(A)$. By the linearity of $\tr$ and $\LL$, we have
\begin{equation}\label{eq:signminusz}
	\begin{aligned}
		&\phi(\sign(A) ) - \phi(Z^*) \\=& \tr\left( \LL \left((A-\beta \bar{A}) \circ (\sign(A)-Z^*) \right) \right) - \sum\limits_{i=d+1}^{n}\lambda_i (\LL(A\circ \sign(A) )) +\sum\limits_{i=d+1}^{n}\lambda_i (\LL(A\circ Z^*)). 
	\end{aligned}
\end{equation}
Since $Z^*\in\snA$, we have $ \sign(A)  - Z^* \ge 0$; since $\beta \ge 1$ and $A\le~\bar{A}$ (see~\eqref{eq:Aconstruct}), 
we have $A -~\beta \bar{A} \le~0$. Thus $(A-\beta \bar{A}) \circ ( \sign(A) -Z^*) \le~0$. 
In addition, by the fact that $Z^*\neq  \sign(A) $, there exists $(\bar i,\bar j)\in \mathrm{supp}(A)$ such that $( \sign(A)-Z^*)_{\bar i \bar j} > 0$ and $(A-\beta \bar{A})_{\bar i\bar j}<0$.
Therefore,
\begin{equation}
	\label{eq:fact1}
	\tr\left( \LL \left((A-\beta \bar{A}) \circ ( \sign(A)-Z^*) \right) \right) = \sum_{i,j=1}^n\left((A-\beta \bar{A}) \circ ( \sign(A)  -Z^*) \right)_{ij}< 0.
\end{equation}
$\sign(A) - Z^*\ge 0 $ also implies that $ \LL({A}\circ  \sign(A)) \succeq \LL(A\circ Z^*)
$ by Lemma \ref{lem:b1}.
Hence we have $\lambda_i (\LL({A}\circ\sign(A))) \ge
\lambda_i (\LL(A\circ Z^*)) $ for all $1\le i \le n$ by Weyl's inequality~\cite{weyl1912asymptotische}. 
Recalling~\eqref{eq:signminusz} and~\eqref{eq:fact1},  we have $\phi(\sign(A)) - \phi(Z^*)<~0$. This contradicts the optimality of $Z^*$.

Therefore, $Z^* = \sign(A) $ is the unique optimal solution to problem $\min_{Z\in\snA\cap [0,1]^{n\times n}} \phi(Z)$, which completes the proof of (a).

(b) Let $$\overline{\beta}=\frac{1}{\alpha\sum_{i,j}A_{ij}}\min \left\{\lambda_{+}(\mathcal{L}({A} \circ Z)) \mid Z \in \snA \cap\{0,1\}^{n \times n}, \mathcal{L}({A} \circ Z) \neq 0\right\},$$ where $\alpha:=p\max\{1, \sum_{i,j}A_{ij}/(n^2\min_{(i,j)\in\mathcal{J}}A_{ij})\}$, and $\lambda_{+}(\mathcal{L}({A} \circ Z)) $ is the smallest positive eigenvalue of $\mathcal{L}({A} \circ Z)$. We prove the desired result with $0<\beta<\overline{\beta}$.
Assume for contradiction that $\rank(\LL(A\circ Z^*))>n-d$. 
Since the support set of $A\circ Z^*$ equals to that of $A\circ \lceil Z^*\rceil$,
we know from Lemma \ref{lem:b2} that 
$\rank(\LL(A\circ \lceil Z^*\rceil )) = \rank(\LL(A\circ Z^*))>n-d.$
By $(H^*)\zz H^*=I_{d}$, the positive semi-definiteness of $\LL({A}\circ
\lceil{Z^*}\rceil)$ (Lemma \ref{lem:b1}), and the Rayleigh-Ritz theorem, we know that $\tr\left((H^*)\zz \LL({A}\circ \lceil{Z^*}\rceil)
H^*\right) \ge \lambda_{d}(\LL({A}\circ \lceil{Z^*}\rceil))>0 $.
Recalling the definition of $\overline{\beta}$, we obtain
$$
\tr\left((H^*)\zz \LL({A}\circ \lceil{Z^*}\rceil)
H^*\right) \ge \alpha \overline{\beta} \sum_{i,j}A_{ij} .$$
It then yields
\begin{align}
	\notag
	f(\lceil{Z^*}\rceil, \, H^* ) 
	= & \tr\left((H^*)\zz \LL(A\circ \lceil{Z^*}\rceil)H^*\right) - \beta\tr(\bar{A}\lceil{Z^*}\rceil) \ge  \alpha \overline{\beta} \sum_{i,j}A_{ij}- \beta \tr(\bar{A}\lceil{Z^*}\rceil) 
	\\ 
	\ge &  \alpha\overline{\beta} \tr({A}\lceil{Z^*}\rceil) - \beta \tr(\bar{A}\lceil{Z^*}\rceil) 
	= \tr(\alpha\overline{\beta} {A}\lceil{Z^*}\rceil- \beta\bar{A}\lceil{Z^*}\rceil)
	 \ge 
	0 =f(0, H^* ),
	\label{eq:fge}
\end{align}
where the second inequality is due to $Z^*\in\{0,1\}^{n\times n}$, and the last one comes from 
$\bar{A}\leq \alpha A$ and $\beta <  \overline{\beta}$.
By the global optimality of $(Z^*,H^*)$ and~Theorem \ref{thm:x}, $(\lceil{Z^*}\rceil, H^* )$ is also a global optimal solution to problem~\eqref{eq:model}. We then obtain from \eqref{eq:fge} that {$(0, H^* ) $} is also a global optimal solution to problem~\eqref{eq:model}.   This implies $0= \rank(\LL(A\circ 0))\geq n-d$ by Theorem~\ref{thm:thmor}, which contradicts with our assumption $d< n$. Recalling Theorem~\ref{thm:thmor}(b), we  complete the proof of part (b).
\end{proof}

\section{Comparisons of ACC, NMI and RMV on TDT2}\label{sec:overall}

In this section, we compare the performance of six methods tested in Subsection~\ref{sec:cpu} on the full TDT test set using ACC, NMI, and RMV.
The numerical results are shown in  Tables~\ref{tab:TDTacc}--\ref{tab:TDTrmv}. {We adopt the same numerical settings as in Subsection~\ref{sec:cpu}.}
From these tables, we observe that in most cases, COSSC achieves the highest ACC and NMI values across all methods in this experimental setup, while its RMV remains consistently zero.

\begin{table}[H]
\centering
\setlength{\tabcolsep}{4pt}
\caption{Comparison of Overall Performance on ACC.}\label{tab:TDTacc}
\begin{tabular}{|c|c|c|c|c|c|c|c|c|c|c|}
	&2&3&4&5&6&7&8&9&10\\
	\hline
	heuristic $k$-means &0.876 &0.754 &0.757 &0.723 &0.657 &0.687 &0.640 &0.626 &0.639\\
	constrained $k$-means &0.893 &0.862 &0.887 &0.830 &0.827 &0.788 &0.788 &0.723 &0.704\\ 
	SCA &  \textbf{0.998} &\textbf{0.997} &0.958 &0.966 &0.929 &0.915 &0.931 &0.853 &0.928\\
	SymNMF & 0.997 &0.997 &0.939 &0.918 &0.930 &0.890 &0.913 &0.896 &0.909\\
	SemiNMF  &0.979 &0.947 &0.948 &0.933 &0.891 &0.893 &0.828 &0.911 &0.831\\
	COSSC & 0.991 &\textbf{0.997} &\textbf{0.986} &\textbf{0.989} &\textbf{0.964} &\textbf{0.987} &\textbf{0.952} &\textbf{0.956} &\textbf{0.955}
\end{tabular}
\end{table}
\vspace{-5mm}
\begin{table}[H]
\centering
\setlength{\tabcolsep}{4pt}
\caption{Comparison of Overall Performance on NMI.}\label{tab:TDTmut}
\begin{tabular}{|c|c|c|c|c|c|c|c|c|c|c|}
	&2&3&4&5&6&7&8&9&10\\
	\hline
	heuristic $k$-means &0.685 &0.650 &0.747 &0.738 &0.700 &0.756 &0.718 &0.716 &0.749\\
	constrained $k$-means &0.900 &0.833 &0.767 &0.767 &0.800 &0.800 &0.867 &0.867 &0.800\\ 
	SCA &  \textbf{0.984} &0.973 &0.929 &0.935 &0.910 &0.901 &0.905 &0.867 &0.925\\
	SymNMF & 0.978 &0.974 &0.890 &0.893 &0.897 &0.887 &0.889 &0.877 &0.900\\
	SemiNMF  &0.893 &0.833 &0.886 &0.906 &0.857 &0.879 &0.814 &0.891 &0.843\\
	COSSC & 0.980 &\textbf{0.980} &\textbf{0.946} &\textbf{0.971} &\textbf{0.942} &\textbf{0.966} &\textbf{0.933} &\textbf{0.932} &\textbf{0.943}
\end{tabular}
\end{table}

\begin{table}[H]
\centering
\caption{Comparison of Overall Performance on RMV.}\label{tab:TDTrmv}
\begin{tabular}{|c|c|c|c|c|c|c|c|c|c|c|}
	&2&3&4&5&6&7&8&9&10\\
	\hline
	heuristic $k$-means &0.058&0.236&0.330&0.360&0.464&0.490&0.591&0.708&0.672\\
	constrained $k$-means &  0 & 0 &0 &0 &0 &0 &0 &0 &0\\ 
	SCA &  0 & 0 &0 &0 &0 &0 &0 &0 &0 \\ 
	SymNMF &0&0.051&0.066&0.079&0.163&0.117&0.152&0.169&0.202\\
	SemiNMF  & 0 & 0 &0 &0 &0 &0 &0 &0 &0\\
	COSSC & 0 & 0 &0 &0 &0 &0 &0 &0 &0  \\
\end{tabular}
\end{table}

\bibliographystyle{plain}
\bibliography{ref_comscp.bib}


\end{document}